\documentclass[]{article}
\usepackage[
textheight=8.9in,
textwidth=5.5in,
top=1in,
headheight=12pt,
headsep=25pt,
footskip=30pt
]{geometry}
\usepackage[utf8]{inputenc}
\usepackage[T1]{fontenc}
\usepackage{hyperref}
\usepackage{xr-hyper}
\usepackage{url}
\usepackage{booktabs}
\usepackage{nicefrac}
\usepackage{microtype}
\usepackage{subcaption}

\usepackage{amsmath, amsthm, amsfonts, amssymb}
\usepackage[dvipsnames]{xcolor}
\usepackage{stackengine}
\usepackage[makeroom]{cancel}
\usepackage{graphicx}
\usepackage{enumitem}

\newcommand{\differential}{{\rm{d}}}
\newcommand{\hess}{\rm{Hess}}
\newcommand\dtheta{\differential\theta}
\newcommand\dr{\differential r}
\newcommand\dmu{\differential \mu}
\newcommand{\dotp}[2]{\left\langle#1, #2\right\rangle}
\newcommand\braketm[2]{\left\langle #1, #2\right\rangle}
\newcommand\braket[2]{\left\langle #1, #2\right\rangle}
\newcommand{\braketd}[2]{\left\langle #1, #2\right\rangle_*}
\newcommand{\braketf}[2]{\left\langle#1,#2\right\rangle_\Fcal}

\newcommand{\Exp}{\mathbb E}
\newcommand{\Ecal}{\mathcal E}
\newcommand{\Fcal}{\mathcal F}
\newcommand{\Vcal}{\mathcal V}
\newcommand{\Pcal}{\mathcal P}
\newcommand{\Kcal}{\mathcal K}
\newcommand{\Mcal}{\mathcal M}
\newcommand{\Mac}{\mathcal{M}_{\rm{ac}}}

\newcommand{\Tcal}{\mathcal T}

\newcommand{\diff}{{\rm{d}}}
\newcommand{\tr}{{\rm{trace}}}
\newcommand{\Rbb}{\mathbb R}
\newcommand{\xtni}{X^{i}_t}

\newcommand{\mutn}{\mu_{t}^{n}}

\newcommand\vect[1]{\left(\begin{array}{c} #1 \end{array} \right)}

\newtheorem{lemma}{Lemma}
\newtheorem{proposition}{Proposition}

\newtheorem{theorem}{Theorem}
\newtheorem*{theorem*}{Theorem}

\newcommand\citen[1]{\cite#1}
\author{%
  Walid Krichene\thanks{Google. \texttt{walidk@google.com}}
\and
  Kenneth F.~Caluya\thanks{Department of Applied Mathematics, University of California, Santa Cruz. \texttt{kcaluya@ucsc.edu, ahalder@ucsc.edu}}
\and
  Abhishek Halder\footnotemark[2] \\
}
\date{}
\title{Global Convergence of Second-order Dynamics in Two-layer Neural Networks}

\begin{document}

\maketitle
\begin{abstract}
Recent results have shown that for two-layer fully connected neural networks, gradient flow converges to a global optimum in the infinite width limit, by making a connection between the mean field dynamics and the Wasserstein gradient flow. These results were derived for first-order gradient flow, and a natural question is whether second-order dynamics, i.e., dynamics with momentum, exhibit a similar guarantee. We show that the answer is positive for the heavy ball method. In this case, the resulting integro-PDE is a nonlinear kinetic Fokker Planck equation, and unlike the first-order case, it has no apparent connection with the Wasserstein gradient flow. Instead, we study the variations of a Lyapunov functional along the solution trajectories to characterize the stationary points and to prove convergence. While our results are asymptotic in the mean field limit, numerical simulations indicate that global convergence may already occur for reasonably small networks.
\end{abstract}

\section{Introduction}
The empirical success of neural network models has prompted several theoretical studies that attempt to shed some light on their performance, and provide guarantees under suitable assumptions. For fully connected networks, universal approximation results such as~\cite{cybenko1989approximation,barron1993universal} provide a partial explanation for this empirical success, by proving that a large enough network can approximate any continuous function on a compact set, though such results do not address the dynamics of learning, i.e., whether local search algorithms such as gradient descent can find global solutions. Recent works~\cite{mei2018mean,chizat2018global} have tackled this question for two-layer networks, and proved convergence to global solutions, by studying the dynamics in the space of distributions over parameters. They make the observation that (Euclidean) gradient flow in the parameter space is equivalent to a Wasserstein gradient flow in the distribution space. This allows for an analysis of the long-time behavior of the dynamics in the mean field limit, i.e., when the width of the network tends to infinity.

To the best of our knowledge, previous works in this setting, such as~\cite{mei2018mean,chizat2018global,sirignano2018mean,rotskoff2018parameters}, have only considered first-order gradient dynamics, and a natural question is whether similar guarantees hold for second-order dynamics, i.e., dynamics with momentum. This is the subject of our investigation. Momentum methods such as the heavy ball method~\cite{polyak1964some}, Nesterov's method~\cite{nesterov1983method}, or the Adam method~\cite{kingma2014adam}, are widely used in practice~\cite{sutskever2013importance} and have received significant attention in the optimization literature. Their continuous-time counterpart is given by a family of second-order differential equations, which can be interpreted as damped nonlinear oscillators~\cite{attouch2000heavy,cabot2009long,goudou2009gradient,gadat2018stochastic}. For example,~\cite{su2014differential,wibisono2016variational} studied the continuous-time limit of Nesterov's method, which is an instance in this family with a particular form of damping.

In this paper, our analysis will focus on the heavy ball method--perhaps the simplest and the earliest instance of second-order optimization dynamics. It corresponds to a constant damping coefficient, making the analysis more tractable. Even in this relatively simple setting, the distribution dynamics for a two-layer neural network is given by a nonlinear kinetic Fokker-Planck equation~\cite{villani2009hypocoercivity}, and unlike the first-order case, there is no apparent connection with the Wasserstein gradient flow. Hence, our approach to analyze the mean field dynamics will be somewhat different, even though the tools we use are similar. Our analysis takes inspiration from previous works in the first-order case~\cite{mei2018mean,chizat2018global}, and also from the study of kinetic Fokker-Planck equations~\cite{bouchut1995long,soler1997asymptotic}.

\subsection{Two-layer neural networks}
\label{sec:network}
We describe the problem setting before summarizing our results. We seek to learn a function $\psi \in \Fcal$, where $\Fcal$ is a Hilbert space equipped with the inner product $\braketf{\cdot}{\cdot}$. The model is parameterized by $(\theta_1, \dots, \theta_n) \in \Theta^n$, and its output is given by
\begin{equation}
\label{eq:network}
\psi_{\theta_1, \dots, \theta_n}(x) := \frac{1}{n}\sum_{i = 1}^n \Psi(\theta_i)(x),
\end{equation}
where $n$ is the number of neurons (also referred to as the width of the network), $x$ is the input vector, and $\Psi(\theta_i) \in \Fcal$. In the two-layer neural network setting, we take $\Theta = \Rbb^{d}$, $\theta_i = (a_i, b_i) \in \Rbb^{d-1} \times \Rbb$, and $\Psi(\theta_i)(x) = b_i s (\dotp{a_i}{x})$, where $a_i, b_i$ are the weights of the first and the second layer, respectively, and $s: \Rbb \mapsto \Rbb$ is an activation function. While~\eqref{eq:network} is perhaps an unusual way to describe the output of a neural network, it highlights a structure that lends itself to mean field analysis: the model can be viewed as an average of ``basis functions'' $\Psi(\theta)(\cdot)$, parameterized by the vector $\theta \in \Theta = \Rbb^d$. This point of view allows us to further rewrite $\psi$ as the integral
\begin{equation}
\label{eq:network_2}
\psi_\mu = \braketm{\Psi}{\mu} := \int_\Theta \Psi(\theta)\:\differential\mu(\theta),
\end{equation}
where $\mu$ is a probability measure on $\Theta$, encoding the parameter distribution.
When $\mu$ is an average of Dirac masses, i.e., $\mu = \frac{1}{n} \sum_{i = 1}^n \delta_{\theta_i}$, the integral \eqref{eq:network_2} reduces to the summation~\eqref{eq:network}.

We are given a convex, Fr\'echet differentiable functional $R: \Fcal \mapsto \Rbb_+$, referred to as the risk functional, which measures the expected loss of the model. For instance, in the quadratic loss case, $R(\psi) = \frac{1}{2}\mathbb E_{(x, y)} |\psi(x) - y|^2 = \frac{1}{2}\braket{\psi - y}{\psi - y}_\Fcal$, where the inner product is taken to be $\braketf{\psi}{\phi} := \mathbb E_{(x, y)} [\psi\phi]$, and the features and labels $(x, y)$ are sampled from a data distribution $D$. We are also given a regularization function $g : \Theta \to \Rbb_+$, and we consider the regularized risk $R(\Psi_{\theta_1, \dots, \theta_n}) + \frac{1}{n} \sum_{i = 1}^n g(\theta_i)$. While this is, in general, a non-convex function of $\theta$, when lifted to the space of probability measures, it becomes
\begin{equation}
\label{eq:F}
F(\mu) := R(\braketm{\Psi}{\mu}) + \braketm{g}{\mu}.
\end{equation}
The functional $F$ is convex and Fr\'echet differentiable. Thus, the learning problem can be recast as a measure-valued convex optimization problem, $\inf_{\mu} F(\mu)$. This is the point of view taken in~\cite{chizat2018global,mei2018mean}, as well as in earlier works such as~\cite{bengio2006convex,bach2017breaking}. We note that while our motivation is the study of neural network dynamics, this setting applies to other problems, see~\cite{bach2017breaking,chizat2018global}.

\subsection{Particle and distribution dynamics}
We give an informal overview of the general strategy used to study the mean field limit.
The first step is to make a connection between the dynamics of the particles $(\theta_1, \dots, \theta_n)\in\Theta^{n}$, and the dynamics of their distribution $\mu$ supported over $\Theta$. Suppose the particles move following a (time-varying) vector field $v_t: \Theta \mapsto \mathbb R^d$, where $t$ indexes time. That is, the trajectories are solutions to the following ordinary differential equation (ODE):
\begin{equation}
\label{eq:ODE}
\dot \theta_t = v_t(\theta_t).
\end{equation}
Then at time $t$, the distribution $\mu_t$ of these particles  -- more precisely, the push-forward of an initial distribution~$\mu_0$ by the flow of the ODE~\eqref{eq:ODE} -- is given by the solution to a partial differential equation (PDE) known as the continuity equation:
\begin{equation}
\label{eq:PDE}
\partial_t \mu_t + \nabla \cdot (\mu_t v_t) = 0,
\end{equation}
where $\nabla \cdot$ stands for the divergence operator. When $\mu_t$ does not have a density,~\eqref{eq:PDE} should be interpreted distributionally.
Let $\mu^n_t$ be the solution of~\eqref{eq:PDE} initialized at $\frac{1}{n} \sum_{i = 1}^n \delta_{\theta_i}$. One can show that if the initial positions $(\theta_1, \dots, \theta_n)$ are drawn from a fixed distribution $\mu_0$, then as the number of particles $n$ tends to infinity, the solutions $\mu^n_{t}$, weakly converge to the solution $\mu_t$ of~\eqref{eq:PDE} initialized at~$\mu_0$. One can then focus on studying the dynamics of the mean field limit $\mu_t$.

The vector field $v_t$ in~\eqref{eq:ODE} and~\eqref{eq:PDE}, which describes the movement of the particles, is determined by the particular learning dynamics under consideration. The choice of $v_t$ in~\cite{mei2018mean,chizat2018global,sirignano2018mean,rotskoff2018parameters} corresponds to the first-order gradient flow in $\theta$. In this case, the PDE~\eqref{eq:PDE} has additional structure: it corresponds to a Wasserstein gradient flow of the functional~$F$ in \eqref{eq:F}, as is well known in the optimal transportation literature~\cite{jordan1998variational,ambrosio2008gradient}. We study a different choice of $v_t$ that corresponds to the heavy ball method. Our primary goal is to provide, under suitable assumptions, convergence guarantees to the global minimizers of $F$, which is similar in spirit with~\cite{chizat2018global,mei2018mean}. The latter two results differ: in~\cite{chizat2018global}, the authors study \emph{deterministic} gradient flow under a homogeneity assumption on the objective, and prove that \emph{assuming $\mu_t$ converges}, it can only converge to a global minimum. In~\cite{mei2018mean}, the authors study \emph{noisy} gradient flow for the quadratic loss, and prove that $\mu_t$ converges arbitrarily close to the global minimum (depending on the magnitude of the noise). Our approach is closer to the latter: we study convergence of the \emph{noisy} heavy ball method. The convergence of the noiseless second-order dynamics in the mean field limit remains an open question.

\subsection{Summary of contributions}
We start by deriving the distributional PDE associated with the heavy ball method (Section~\ref{sec:second_order}). To study the dynamics in the mean field limit, we define a Lyapunov functional in Section~\ref{sec:LyapunovFunctional} and bound its variations along solution trajectories. This relies on a general criterion given in Lemma~\ref{lem:variation}, which reveals a close connection between the Lyapunov functions for dynamics with no particle interaction (as in convex optimization) and the Lyapunov functionals for mean field dynamics.

Equipped with this result, in Section \ref{sec:StationarySolution}, we characterize the stationary solutions in Theorem~\ref{thm:decomposition}, and show that they must satisfy a Boltzmann fixed point equation, for which we prove the existence and uniqueness of a solution in Proposition~\ref{prop:boltzmann}. Furthermore, we show in Theorem~\ref{thm:stationarymeasure} that the solution trajectory converges to this unique stationary point. Finally, we show in Theorem~\ref{thm:minimum} that by using vanishingly small noise, the limit can be made arbitrarily close to the global infimum of $F$.

In Section \ref{sec:Numerics}, we illustrate these results with numerical experiments that include other variants of second-order dynamics beyond the heavy ball method. The experiments suggest that the convergence may already occur with a reasonably small number of particles.

The proofs are deferred to the appendix.

\paragraph{Notation}

We denote the parameter space by $\Theta = \Rbb^d$, and its tangent bundle by $\Tcal\Theta = \Rbb^d \times \Rbb^d$. We use $\Mcal\left(\mathcal{T}\Theta\right)$ to denote the space of probability measures on $\mathcal{T}\Theta$, and $\Mac\left(\mathcal{T}\Theta\right) \subset \Mcal\left(\mathcal{T}\Theta\right)$ to denote the space of probability measures that are absolutely continuous w.r.t. the Lebesgue measure. We identify $\mu\in\Mac(\mathcal{T}\Theta)$ with its density $\rho$ using the relation $\differential\mu = \rho(\theta,r) \differential\theta\differential r$, and denote the space of corresponding density functions as $\Pcal\left(\mathcal{T}\Theta\right)$. When there is no potential confusion, we ease the notation by writing $\Mcal,\Pcal$ in lieu of $\Mcal\left(\mathcal{T}\Theta\right),\Pcal\left(\mathcal{T}\Theta\right)$.
We use the symbol $\langle\cdot,\cdot\rangle$ for inner products. When the arguments of $\langle\cdot,\cdot\rangle$ are finite dimensional vectors, it denotes the Euclidean inner product, and when the arguments are scalar-valued functions, it denotes the integral of the product of these functions w.r.t. the Lebesgue measure. Given two vector fields $u, v$ on $\Tcal\Theta$, we use $\braket{u}{v}_*$ to denote the integral $\int_{\mathcal{T}\Theta}\dotp{u(\theta,r)}{v(\theta,r)}\differential\theta\differential r$.
We use $|\cdot|$, $\|\cdot\|_\Fcal$ and $\|\cdot\|_p$ to denote the Euclidean norm on $\Rbb^d$, the Hilbert norm on $\Fcal$, and the $L^p$ norm, respectively. The Fr\'echet differential of a functional $F: \Fcal \mapsto \Rbb$ is denoted by $F'$, and the (Euclidean) gradient of a function $f : \mathcal{T}\Theta \mapsto \Rbb$ is denoted by $\nabla f$.

\section{Mean field second-order dynamics}
\label{sec:second_order}

\subsection{Assumptions}
Let $F_0(\mu) := R(\braket{\Psi}{\mu})$ denote the unregularized loss, and $F(\mu) = F_0(\mu) + \braket{g}{\mu}$. We make the following assumptions:
\begin{enumerate}[label=(A\arabic*)]
\item\label{A:R_diff} $R: \Fcal \mapsto \Rbb_+$ is convex, Fr\'echet differentiable.
\item\label{A:psi_diff} $\Psi : \Theta \mapsto \Fcal$ is Fr\'echet differentiable.
\item\label{A:grad_bounded} $\nabla F'(\rho) \in L^\infty(\Theta)$ for all $\rho \in \Pcal(\Theta)$.
\item\label{A4} $\{F'_0(\rho) : \rho \in L^{1}(\Theta), \|\rho\|_1 \leq 1\}$ is uniformly equicontinuous and uniformly bounded, and $g: \Theta \to \Rbb_+$ is differentiable and confining, i.e. $\lim_{|\theta| \to \infty} g(\theta) = \infty$ and $\exp (-\beta g)$ is integrable for all $\beta > 0$.
\end{enumerate}
We discuss some of the implications of these assumptions. \ref{A:R_diff} and \ref{A:psi_diff} are basic regularity assumptions implying that $F$ is Fr\'echet differentiable and $F'(\rho) : \Theta \to \Rbb$ is a differentiable function of $\theta$, so the gradient of the loss in the parameter space is well-defined. \ref{A:grad_bounded} is used to prove the existence and uniqueness of solutions of the PDE. The assumption that the loss is regularized by $g$, together with the condition \ref{A4}, are important to guarantee existence of a stationary solution, as is common in the literature~\cite{bouchut1995long,pavliotis2014stochastic,mei2018mean}. In particular, the assumption that the regularizer $g$ is confining is rather mild; it essentially requires $g(\theta)$ to grow sufficiently fast when $|\theta|$ tends to infinity. One simple choice is to take $g(\theta) = |\theta|$ (except near $0$, since we also require differentiability of $g$). In Appendix~\ref{app:quadratic}, we explicate the foregoing assumptions in the context of the quadratic loss, and show that they are implied by the assumptions made in previous work.

\subsection{Second-order dynamics}
Given a differentiable $f: \Rbb^d \mapsto \Rbb$, a broad family of second-order dynamics is described by the~ODE
\begin{equation}
\label{eq:second_order_ode}
\ddot \theta_t = -\nabla_\theta f(\theta) - \gamma_t \dot\theta_t,
\end{equation}
which can be interpreted as a dissipative nonlinear oscillator with potential $f$, and damping coefficient~$\gamma_t$, see~\cite{haraux1991systemes,attouch2000heavy,cabot2009long,su2014differential}. Under certain assumptions, such as convexity of $f$, it can be shown that the solutions converge to global minimizers of $f$, see~\cite{cabot2009long}. When $\gamma_t \equiv \gamma/t$ for some positive constant $\gamma$, this corresponds to Nesterov's method in continuous-time~\cite{su2014differential}, and when $\gamma_t \equiv \gamma$ is a time-independent positive constant, it corresponds to the heavy ball method~\cite{attouch2000heavy,gadat2018stochastic}.

As will become clear shortly, the potential $f$ in our setting is time-varying due to the interaction between particles. Recall from~\eqref{eq:F} that the objective functional is $F(\mu)=R(\braketm{\Psi}{\mu}) + \braket{g}{\mu}$, where $\mu$ is a distribution over parameters $\theta\in\Theta$. The gradient of the objective in the parameter space $\Theta$ is given by $\nabla F'(\mu)(\theta)$, see Appendix~\ref{app:gradient} for a detailed derivation. In particular, setting $v_t(\theta) \equiv - \nabla F'(\mu_t)(\theta)$ in ~\eqref{eq:ODE}-\eqref{eq:PDE} corresponds to the first-order gradient flow, as in~\cite{chizat2018global,mei2018mean}.

In the second-order case, it is convenient to write~\eqref{eq:second_order_ode} as a system of two first-order equations describing the evolution of position-velocity pair $(\theta,r)\in\mathcal{T}\Theta$ (the tangent bundle). Then $\mu_t \in \Mcal(\Tcal\Theta)$ denotes the joint distribution over $\Tcal\Theta$ at time $t$, and $[\mu_t]^\theta$ is the corresponding marginal measure over $\Theta$. We suppose that a Brownian motion is applied to the velocity (or rate) $r$, resulting in the following It\^{o} stochastic differential equation (SDE):
\begin{equation}
\label{eq:second_order_particle}
\differential \vect{\theta \\ r}  = \vect{r \\ - \nabla F'([\mu_t]^\theta)(\theta) - \gamma r}\differential t + \vect{0 \\ \sqrt{2\gamma\beta^{-1}}\:\differential W_t},
\end{equation}
where the parameter $\beta > 0$ is referred to as the inverse temperature, $\gamma > 0$ is the constant damping coefficient, and $W_t$ is the standard Wiener process in the tangent space of $\Theta$. Eq.~\eqref{eq:second_order_particle} is an underdamped Langevin equation with interaction potential $F'([\mu_t]^\theta)$. It describes the stochastic heavy ball method~\cite{gadat2018stochastic} in the parameter space. The dependence of the potential on $\mu_t$ reflects the fact that the output of the neural network (and its loss) depend not on a single particle, but on the distribution of all particles. Note that the dependence is on the marginal $[\mu_t]^\theta$, since the loss only depends on positions, and not velocities.

The distribution dynamics corresponding to~\eqref{eq:second_order_particle} is given by
\begin{equation}
\label{eq:PDE_second_order}
\partial_t \mu_t = -\nabla . \left[\mu_t . \vect{r \\ -\nabla F'([\mu_t]^\theta) - \gamma r} \right] + \gamma\beta^{-1}\Delta_r \mu_t,
\end{equation}
where $\Delta_r$ denotes the Laplacian operator w.r.t. the $r$ variable, and corresponds to the Brownian motion applied to $r$. The integro-PDE~\eqref{eq:PDE_second_order} is a nonlinear kinetic Fokker-Planck equation.

\paragraph{Consistency of the mean field limit}
We now provide a consistency result between the second-order particle dynamics and the mean field limit. 
\begin{theorem}
\label{thm:consistency}
Let $\mu_{0}\in\Mcal(\Tcal\Theta)$. Consider a set of $n$ interacting particles with states $\left\{ (\theta^i_t,r^i_t)\right\}_{i=1}^{n}$ where $i$ denotes the $i$-th particle. Suppose these particles solve copies of the SDE~\eqref{eq:second_order_particle} indexed by $i=1,\hdots,n$, in which $\mu_t$ is replaced by the empirical distribution $\mu_{t}^{n}:=\frac{1}{n}  \sum_{i=1}^{n} \delta_{(\theta^i_t,r^i_t)}$, and with initial states $\{(\theta^i_0,r^i_0)\}_{i = 1}^n$ sampled independently from $\mu_0$. Then there exists $\mu_t\in\Mcal(\Tcal\Theta)$ such that, almost surely, $\mu_t^n \to \mu_t$ weakly, as $n \to \infty$. Furthermore, $\mu_t$ solves~\eqref{eq:PDE_second_order} in the weak distributional sense with initial condition $\mu_0$.
\end{theorem}

This result motivates the study of the long-time behavior of the mean field limit~$\mu_t$. One of the advantages of the mean-field setting is that one can work with absolutely continuous distributions, which simplifies the analysis.

\paragraph{Existence and uniqueness of solutions}
We make the following assumption on the initial condition $\mu_0$, both to obtain existence and uniqueness of a solution, and to guarantee finiteness of the free energy, introduced in the next section.
\begin{enumerate}[label=(A\arabic*)]\setcounter{enumi}{4}
\item\label{A:init} $\mu_{0}$ is absolutely continuous, and the associated PDF $\rho_{0}$ satisfies $\langle g(\theta) + |r|^{2}/2,\rho_{0}\rangle < \infty$, $\langle\log^{+}\rho_{0},\rho_{0}\rangle < \infty$, and $\int |\nabla F_0^{\prime}([\rho_0]^{\theta})(\theta)|^2 \differential\theta < \infty$,
\end{enumerate}
where $\log^{+}\rho:= \max\{\log\rho,0\}$. It is known (see~\cite{victory1990classical,bouchut1995long}) that if $F$ satisfies \ref{A:grad_bounded} and the initial condition $\mu_0$ satisfies \ref{A:init}, then (\ref{eq:PDE_second_order}) admits a unique solution $\mu_{t}\in C\left([0,\infty),\Mcal(\Tcal\Theta)\right)$. That $\mu_{t}$ remains absolutely continuous for all $t\geq 0$, and hence $\rho_{t}$ exists, will be proved in Theorem \ref{thm:stationarymeasure}.
Since the solution is absolutely continuous, $\mu_{t}$ and $[\mu_{t}]^{\theta}$ in \eqref{eq:second_order_particle}-\eqref{eq:PDE_second_order} can be replaced by the corresponding PDFs $\rho_{t}$ and $[\rho_{t}]^{\theta}$, respectively.

\subsection{The linear case}
\label{sec:linear}
When $F(\mu) = \braket{f}{\mu}$ for some function $f: \Theta \mapsto \Rbb$, we have $F'(\mu) = f$ and there is no particle interaction in~\eqref{eq:second_order_particle}. While this situation is irrelevant in the neural network context, it is instructive to review results in the linear setting. In this case,~\eqref{eq:PDE_second_order} becomes a \emph{linear} Fokker-Planck PDE, which can be shown (e.g., Proposition 6.1 in~\cite{pavliotis2014stochastic}) to admit a unique stationary solution with PDF $\rho_\infty(\theta, r) = \exp \left(-\beta\left(f(\theta) + \frac{1}{2}|r|^2 \right)\right)/Z$,
where $Z$ is a normalizing constant. Under additional assumptions on the confining potential $f$, one can also study the rate of convergence of $\rho_t$ to $\rho_{\infty}$, see e.g.~\cite{haraux1991systemes,villani2009hypocoercivity,pavliotis2014stochastic,bakry2008rate}. Our situation corresponds to a \emph{nonlinear} kinetic Fokker-Planck equation, which is not well-understood in the general setting. Some special cases have been studied in the literature, such as when the interaction potential $F'(\mu)$ is a convolution~\cite{bouchut1995long,soler1997asymptotic,carrillo2003kinetic,villani2009hypocoercivity}. The convolution structure in these references is motivated from physical dynamics--the electrostatic Coulomb interaction in plasma and semiconductor dynamics \cite{dolbeault1991stationary,abdallah1995child,degond2000mathematical}, and the gravitational Newton interaction in stellar dynamics \cite{chandrasekhar1943stochastic,padmanabhan1990statistical,batt1993rigorous}--and leads to the Valsov-Poisson-Fokker-Planck equations \cite{victory1990classical,perthame2004mathematical,huang2000variational}. Unfortunately, this is not the case in our neural network setting, and these results do not directly apply. However, we will use similar techniques, and will prove that the stationary solutions have a similar characterization.

\section{Variations of a Lyapunov functional}
\label{sec:LyapunovFunctional}
Hereafter, we work with the PDF trajectory $\rho_t\in C\left([0,\infty),\Pcal(\Tcal\Theta)\right)$ associated with the measure-valued solution trajectory $\mu_{t}\in C\left([0,\infty),\Mac(\Tcal\Theta)\right)$ for (\ref{eq:PDE_second_order}). To characterize the stationary solutions, we will study the variations of the following Lyapunov functional, defined for $\rho \in \Pcal(\Tcal\Theta)$,
\begin{equation}
\label{eq:free_energy}
\Ecal(\rho) := F([\rho]^\theta) + \dotp{\frac{1}{2}|r|^2}{\rho} + \beta^{-1} H(\rho),
\end{equation}
where $H(\rho):=\langle\log\rho,\rho\rangle$ is the negative entropy. The functional $\Ecal$ is often referred to as the \emph{free energy}. In this section, we show that along the trajectory $\rho_t$, the functional $\Ecal$ is non-increasing.

\begin{lemma}
\label{lem:variation}
Let $v_t$ be a vector field over $\Tcal \Theta$, and let $\rho_t$ be a solution of the continuity equation $\partial_t \rho_t = - \nabla . (\rho_t v_t)$ with initial condition $\rho_0\in\Pcal(\Tcal\Theta)$. Let $\Vcal : \Pcal(\Tcal\Theta) \mapsto \Rbb$, and suppose that along $\rho_t$, $\Vcal$ is Fr\'echet differentiable and $\Vcal'(\rho_t): \Theta \mapsto \Rbb$ is differentiable. Then for all $t \geq 0$,
\[
\partial_t \Vcal(\rho_t) = \braketd{\nabla \Vcal'(\rho_t)}{\rho_t v_t}.
\]
\end{lemma}
This gives us a simple criterion for a functional $\Vcal$ to be non-increasing along solution trajectories with vector field $v(\rho)$: it suffices that for all $\rho$, the inequality $\dotp{\nabla \Vcal'(\rho)}{v(\rho)} \leq 0$ holds $\rho$-a.e. In the case with no interaction, i.e., $\Vcal(\rho) \equiv \braket{V}{\rho}$ is linear with $V:\Theta\mapsto\mathbb{R}$, and the vector field $v$ is independent of $\rho$, the condition reduces to $\dotp{\nabla V}{v}\leq 0$, which defines Lyapunov functions for single particle dynamics. With this observation, the free energy $\Ecal$ can be viewed as a mean field generalization of $E(\theta, r) := f(\theta) + \frac{1}{2}|r|^2$, which is known to be a Lyapunov function for the heavy ball dynamics, see Appendix~\ref{app:lyap}.

To apply Lemma~\ref{lem:variation} to the free energy $\Ecal$ in \eqref{eq:free_energy}, we use the identity $\Delta_r \rho = \nabla_r \cdot \nabla_r \rho$ to formally rewrite \eqref{eq:PDE_second_order} as follows.\begin{equation}
\label{eq:noisy_second_order}
\partial_t \rho_t = -\nabla. (\rho_t v(\rho_t)),
\quad
v(\rho)(\theta, r) := \vect{r \\ - \nabla F'([\rho]^\theta)(\theta) - \gamma r - \gamma\beta^{-1}\nabla_r \log \rho(\theta, r)}.
\end{equation}
\begin{proposition}
\label{prop:lyap_decreasing}
Consider the Lyapunov functional $\mathcal{E}$ in (\ref{eq:free_energy}). Let $\rho_t$ be the solution to~\eqref{eq:noisy_second_order}. Then
\[
\partial_t \Ecal(\rho_t) = -\gamma \braket{|r + \beta^{-1}\nabla_r \log \rho_t|^2 }{\rho_t}.
\]
\end{proposition}

\begin{proof}
From (\ref{eq:free_energy}), we obtain
\begin{equation}
\label{eq:d_free_energy}
\Ecal'(\rho)(\theta, r) = F'([\rho]^\theta)(\theta) + \frac{1}{2}|r|^2 + \beta^{-1} (1 + \log \rho(\theta, r)),
\end{equation}
and, using the shorthand $\ell_\theta := \beta^{-1}\nabla_\theta \log \rho_t$, $\ell_r := \beta^{-1}\nabla_r \log \rho_t$, we compute
\def\adjmargin{\kern-12pt}
\begin{align*}
\partial_t \Ecal(\rho_t)
&= \braketd{\nabla \Ecal'(\rho_t)}{\rho_t v(\rho_t)} & \text{\adjmargin by Lemma~\ref{lem:variation}}\\
&= \braketd{\vect{\nabla_\theta F'([\rho_t]^\theta) + \ell_\theta \\ r + \ell_r}}{\rho_t\vect{r \\ -\nabla_\theta F'([\rho_t]^\theta) - \gamma r - \gamma\ell_r}} & \text{\adjmargin by~\eqref{eq:noisy_second_order} and~\eqref{eq:d_free_energy}}\\
&= \braket{ - \gamma \dotp{r}{r} - \gamma \dotp{\ell_r}{\ell_r} - 2\gamma \dotp{r}{\ell_r} + \dotp{\ell_\theta}{r} - \dotp{\ell_r}{\nabla_\theta F'([\rho_t]^\theta)} }{ \rho_t}.
\end{align*}
We conclude by showing that the last two terms, $\braket{\braket{\ell_\theta}{r}}{\rho_t}$, and $\braket{\braket{\ell_r}{\nabla_\theta F'([\rho_t]^\theta)}}{\rho_t}$ are equal to zero, using duality of the $\nabla \cdot$ and $\nabla$ operators. The details are provided in Appendix~\ref{app:details}.
\end{proof}
The proposition states that the free energy $\Ecal$ is non-increasing along solution trajectories. This fact, together with additional bounds derived in Appendix~\ref{app:lyap}, are the primary ingredients used to prove our main results in the next section.


\section{Stationary solutions and global convergence}
\label{sec:StationarySolution}

We say $\rho^{\star}\in\Pcal$ is a stationary solution of (\ref{eq:noisy_second_order}) if the solution $\left(\rho_{t}\right)_{t\geq 0}$ obtained with the initial condition $\rho_{0}\equiv \rho^{\star}$, satisfies $\rho_{t} \equiv \rho^{\star}$ for all $t\geq 0$.

In this section, we state our main results (proved in Appendix~\ref{app:stationary}), by characterizing stationary solutions (Theorem~\ref{thm:decomposition}), proving their existence and uniqueness (Proposition~\ref{prop:boltzmann}), and establishing convergence of $\rho_t$, as $t \to \infty$, to the unique stationary point (Theorem~\ref{thm:stationarymeasure}). Furthermore, we show that the limit can be made arbitrarily close to the global infimum (Theorem~\ref{thm:minimum}).

\begin{theorem}
\label{thm:decomposition}
Suppose $\rho^{\star} \in \Pcal(\Tcal\Theta)$ is a stationary solution of~\eqref{eq:noisy_second_order}. Then,
\begin{equation}
\label{eq:sqrt_mu}
\rho^{\star}(\theta, r) = \frac{\exp\left(-\frac{\beta}{2}|r|^2\right)}{Z_1} [\rho^\star]^\theta(\theta).
\end{equation}
where $Z_1$ is the normalizing constant $Z_1 := \int \exp\big(-\frac{\beta}{2}|r|^2\big)\differential r$ and $[\rho^\star]^\theta$ is the $\theta$ marginal. Furthermore, $[\rho^\star]^\theta$ solves the following fixed point equation:
\begin{equation}
\label{eq:boltzmann}
\rho(\theta) = \frac{\exp \left(-\beta F'(\rho)(\theta)\right)}{Z_2(\rho)}, \rho \in \Pcal(\Theta)
\end{equation}
where $Z_2(\rho) := \int \exp \left(- \beta F'(\rho)(\theta)\right) \differential\theta$.
\end{theorem}
The proof crucially relies on the variation of the free energy given in Proposition~\ref{prop:lyap_decreasing}. The theorem states that a stationary solution, if it exists, must be a product distribution, where the $r$ marginal is a Gaussian, and the $\theta$ marginal satisfies a fixed point equation. This product structure is familiar from the linear case (Section~\ref{sec:linear}), where $F'(\rho) \equiv f$ and the RHS of~\eqref{eq:boltzmann} becomes independent of $\rho$, and simply describes a Gibbs distribution. In the nonlinear case, it is not guaranteed, a priori, that \eqref{eq:boltzmann} admits a solution. This is proved in the next proposition; our existence proof invokes Schauder's fixed point theorem~\cite[p. 286, Theorem 11.6]{gilbarg2015elliptic}, and this is where assumption~\ref{A4} comes into play.

\begin{proposition}
\label{prop:boltzmann}
Suppose assumption \ref{A4} holds, and let $T: \Pcal(\Theta) \mapsto \Pcal(\Theta)$ be defined as follows:
\[
T(\rho)(\theta) = \frac{\exp\left(-\beta F'(\rho)(\theta)\right)}{Z_2(\rho)},
\]
where $Z_2(\rho) = \int \exp\left(-\beta F'(\rho)\right)$. Then $T$ has a unique fixed point.
\end{proposition}

We next show that the solution trajectory $\rho_t$ converges to the unique stationary solution $\rho^\star$, under mild assumptions on the initial condition.
\begin{theorem}
\label{thm:stationarymeasure}
Consider a measure $\mu_{0}\in\Mac\left(\mathcal{T}\Theta\right)$ satisfying the assumption \ref{A:init}. Starting from such an initial condition $\mu_{0}$, the solution $(\mu_{t})_{t\geq 0}$ of (\ref{eq:noisy_second_order}) satisfies the following.

\noindent(i) For each $t\geq 0$, the measure $\mu_{t}\in\Mac\left(\mathcal{T}\Theta\right)$, i.e., the associated joint PDF $\rho_{t}\in\Pcal\left(\mathcal{T}\Theta\right)$ exists. 

\noindent(ii) The trajectory $(\rho_{t})_{t\geq 0}$ converges strongly in $L^{1}$ to the unique stationary solution $\rho^{\star}$ of (\ref{eq:noisy_second_order}) as~$t\rightarrow\infty$.
\end{theorem}

Now that we have established the existence and uniqueness of a stationary solution $\rho^\star \in \Pcal(\Tcal\Theta)$, and convergence to $\rho^\star$, we will relate, in the next theorem, $F([\rho^\star]^\theta)$ to $\inf_{\rho \in \Pcal(\Theta)} F(\rho)$. Some intuition can be gained from the linear case: when $F(\rho) = \braket{f}{\rho}$, the stationary solution is simply given by the Gibbs distribution, $[\rho^\star]^\theta(\cdot) \propto \exp(-\beta f(\cdot))$, which concentrates around the minimizers of $f$ as $\beta \to \infty$, thus $F([\rho^\star]^\theta)$ approaches $\inf_{\rho \in\Pcal(\Theta)}F(\rho)$ as $\beta \to \infty$. The same holds in our non-linear setting, as stated in the next theorem.

For $\lambda \in [0, 1]$, let $F_\lambda(\rho) := R(\braket{\Psi}{\rho}) + \lambda \braket{g}{\rho}$, so that $F_1 \equiv F$.
\begin{theorem}
\label{thm:minimum}
Let $\rho^\star$ be the stationary solution of~\eqref{eq:noisy_second_order}, and let $[\rho^\star]^\theta$ be its marginal. Then there exists a constant $C$ that depends on $F$ and $d$, such that for all $\beta \geq 1$,
\[
F_{1 - 1/\beta}([\rho^\star]^\theta) \leq \inf_{\rho \in \Pcal(\Theta)} F_1(\rho) + \frac{C + d \log \beta}{\beta}.
\]
\end{theorem}
The proof of the above theorem has two components: the first is the observation that $\rho^\star$ is a minimizer of the free energy $\Ecal$ (this follows from the characterization in Theorem~\ref{thm:decomposition}), the second is the bounds on the difference between $\Ecal$ and $F$ derived in Appendix~\ref{app:lyap}.

As a consequence of the theorem, the objective value at the stationary point can be made arbitrarily close to the global infimum of $F$ by taking $\beta$ large enough. It is worth emphasizing that the presence of noise, i.e., the diffusion term in \eqref{eq:PDE_second_order}, is essential in guaranteeing existence and uniqueness of the stationary distribution. In the noiseless case, there may exist multiple stationary points that are not global minimizers. The addition of noise can be thought of as an \emph{entropic regularization} of the functional $F(\rho)$, and Theorem~\ref{thm:minimum} says that one can approach the infimum of the unregularized problem in the small noise limit.

\section{Numerical simulations}\label{sec:Numerics}

To illustrate our results, we run synthetic numerical experiments following the setup used in~\cite{chizat2018global}. The model~$\psi$ is a two-layer neural network, as described in Section~\ref{sec:network}, with sigmoid activation function $s(\cdot)$ and width $n$, i.e., $\psi(x) = \frac{1}{n} \sum_{i = 1}^n b_i s(\braket{a_i}{x})$. The features~$x$ are normally distributed in $\Rbb^{d-1}$, and the ground truth labels are generated using a similar neural network $\psi^\star$ with width $n_0$, i.e., $y = \psi^\star(x)$. The risk functional is quadratic, i.e., $R(\psi) = \frac{1}{2}\|\psi(x) - y\|_\Fcal^2 = \frac{1}{2}\mathbb E_{x}[(\psi(x) - \psi^\star(x))^2]$, where the expectation is over the empirical distribution. We implement the stochastic heavy ball method using a simple Euler-Maruyama discretization of~\eqref{eq:second_order_particle}, this will be referred to as (SHB) in the figures. We also implement noiseless, second order dynamics: the heavy ball method, referred to as (HB), and Nesterov's accelerated gradient descent, referred to as (AGD).

\subsection{Convergence to the global infimum}
In a first set of experiments, we set the dimension to $d = 100$, and vary the width~$n$ of the model, while keeping the width of the ground truth network fixed to $n_0 = 20$. No regularization is used in this experiment, so that the model can theoretically achieve zero loss whenever $n \geq n_0$. The results are reported in Figure~\ref{fig:width}.
In the left subplot, each method is run for $10^5$ iterations, and we measure the loss at the last iteration. We repeat the experiment 20 times and plot the average (represented by the lines) and the individual numbers (scatter plot). The right subplot shows the full trajectory for one realization, for the width $n = 100$. The results suggest that the dynamics converge to the global infimum even with a reasonably small width $n$. The results also highlight the effect of noise: the stochastic heavy ball method converges closer to the global minimum when $\beta$ is larger, consistent with Theorem~\ref{thm:minimum}. Finally, the results for the noiseless heavy ball method and Nesterov's method suggest that convergence may occur for a broader class of second-order dynamics than the setting of our analysis.

\begin{figure}[h]
\centering
\includegraphics[width=.44\textwidth]{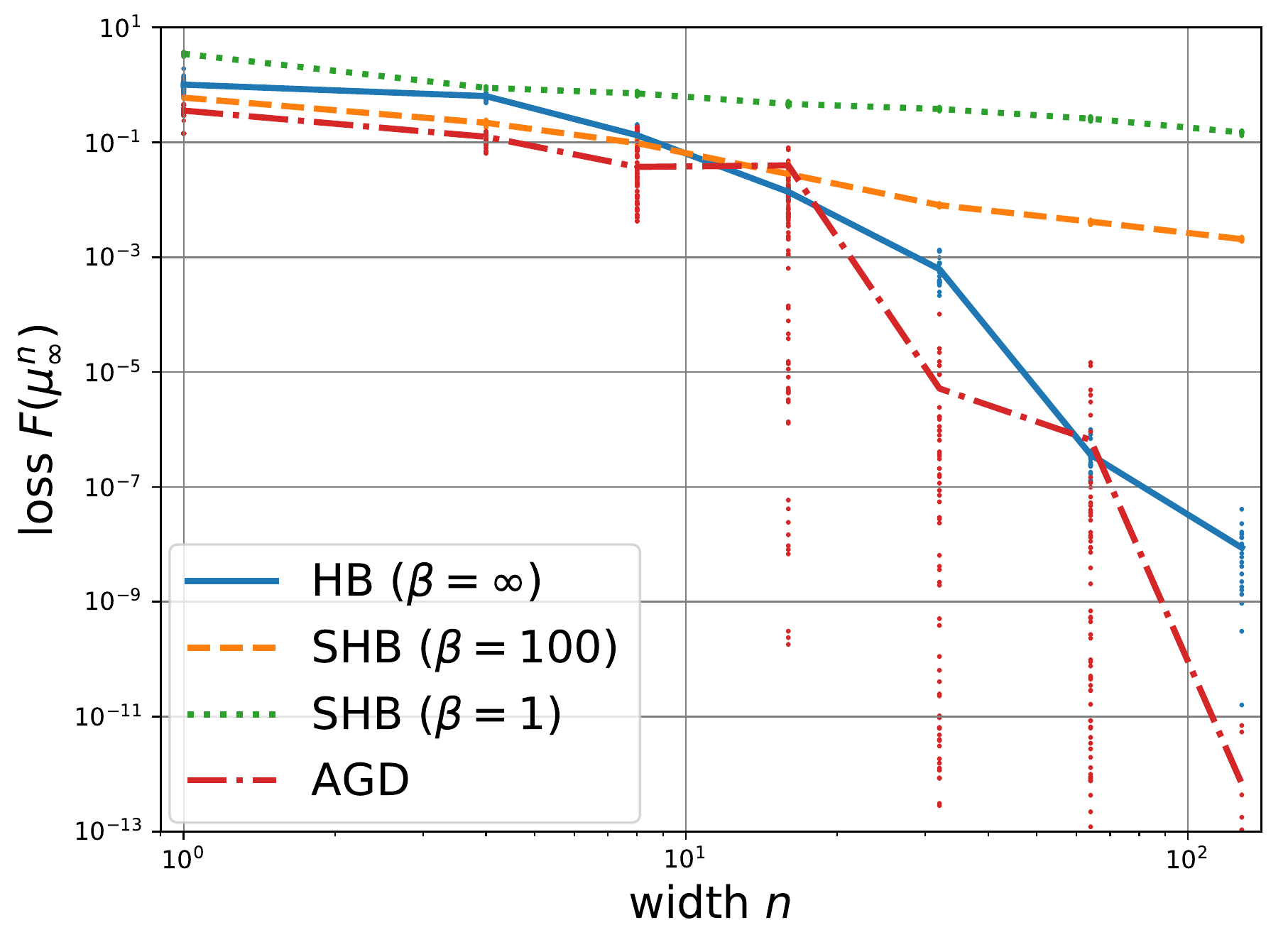}
\hspace{.1\textwidth}
\includegraphics[width=.44\textwidth]{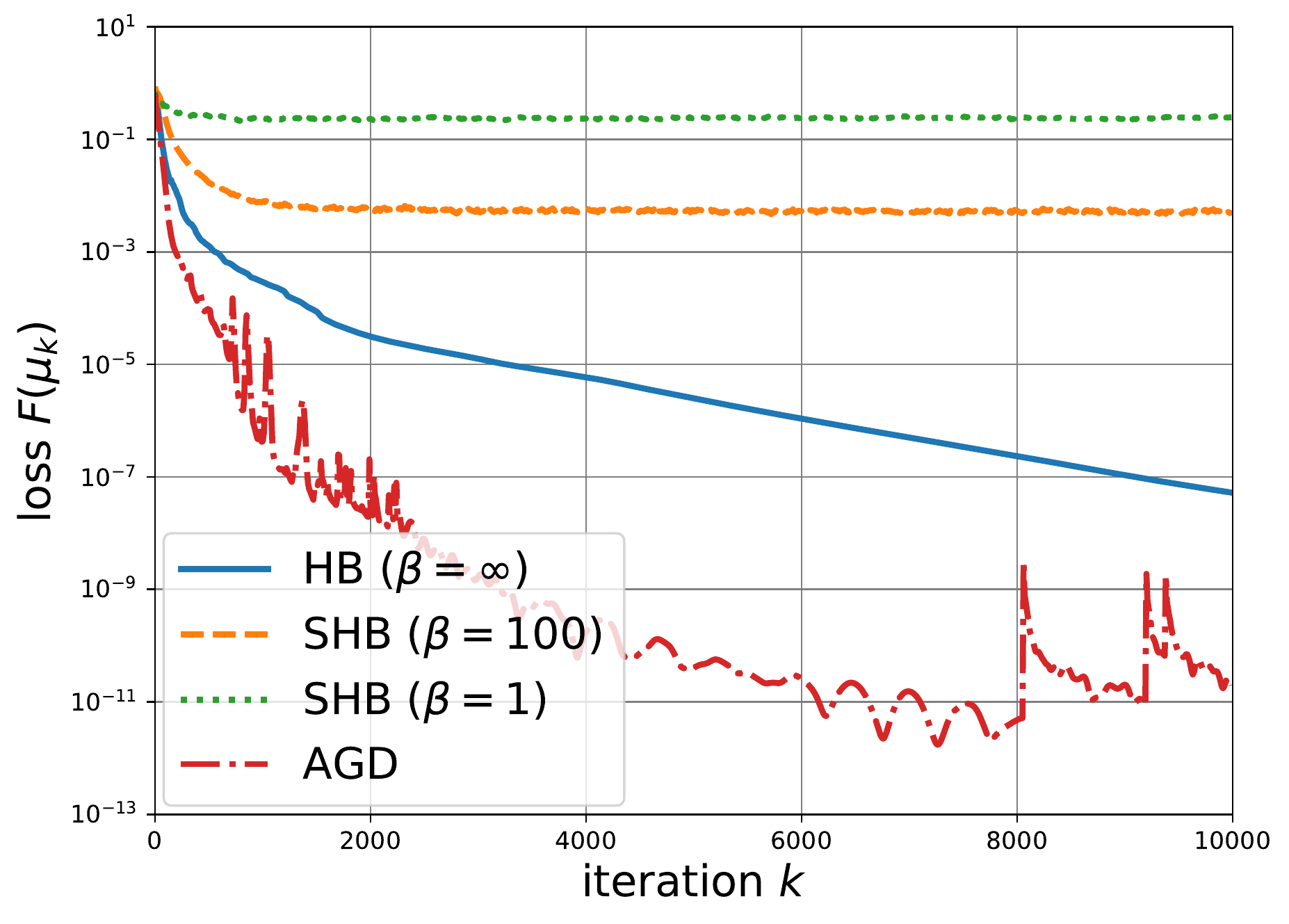}
\caption{Final loss value as the width $n$ of the network increases for several second-order dynamics (left), and sample trajectories for $n = 100$ (right).}
\label{fig:width}
\end{figure}

\subsection{Stationary distribution}
In a second experiment, we illustrate the characterization of the limiting distribution, which according to Theorem~\ref{thm:decomposition}, is the product of its marginals, where the $r$ marginal is a Gaussian $\propto \exp(-\beta|r|^2)$, and the $\theta$ marginal is $\propto \exp(-\beta F'([\rho_\infty]^\theta))$. Recall from~\eqref{eq:F} that $F(\mu) = R(\braket{\Psi}{\mu}) + \braket{g}{\mu}$, thus $F'(\mu)(\cdot) = \braketf{R'(\braket{\Psi}{\mu})}{\Psi(\cdot)} + g(\cdot)$, where $g$ is the regularizer, which we set to $g(\theta) = 0.01 |\theta|$ in this experiment. The risk functional is $R(\psi) = \frac{1}{2}\|\psi - \psi^\star\|_\Fcal^2$, thus $R'(\psi) = \psi - \psi^\star$, and
\begin{equation}
\label{eq:dF_quadratic}
F'([\mu]^\theta)(\cdot) = \mathbb E_{x} \big[ \big( \braket{\Psi}{[\mu]^\theta}(x) - \psi^\star(x) \big) \Psi(\cdot)(x) \big] + g(\cdot).
\end{equation}
In particular, if we apply this expression to the empirical distribution of the particles $[\mu^n]^\theta = \frac{1}{n} \sum_{i = 1}^n \delta_{\theta_i}$, and use $m$ independent samples $x_j \sim D$ to approximate the expectation, we obtain
\begin{equation}
\label{eq:dF_quadratic_approx}
F'([\mu^n]^\theta)(\cdot) \approx g(\cdot) + \frac{1}{m} \sum_{j = 1}^m \Big( \frac{1}{n} \sum_{i = 1}^n \Psi(\theta_i)(x_j) - \psi^\star(x_j) \Big)\Psi(\cdot)(x_j) .
\end{equation}
This gives us an expression of the Boltzmann distribution that we can approximate numerically in the finite particle case, by using $[\mu^n_k]^\theta$ for large $k$, in place of $[\rho_\infty]^\theta$ in $\exp(-\beta F'([\rho_\infty]^\theta))$.

We rerun the same experiment described above, with $n = 200$, $n_0 = 20$, and in lower dimension $d = 2$, so that we can visualize the distributions, and compare the empirical and theoretical marginals at the end of training. The result is shown in Figure~\ref{fig:distributions}, where the empirical marginals (scatter plot) appear to be consistent with the numerical approximation of the Boltzmann distribution (heat map).

\begin{figure}[h]
\centering
\begin{subfigure}{0.34\textwidth}
\includegraphics[width=\textwidth]{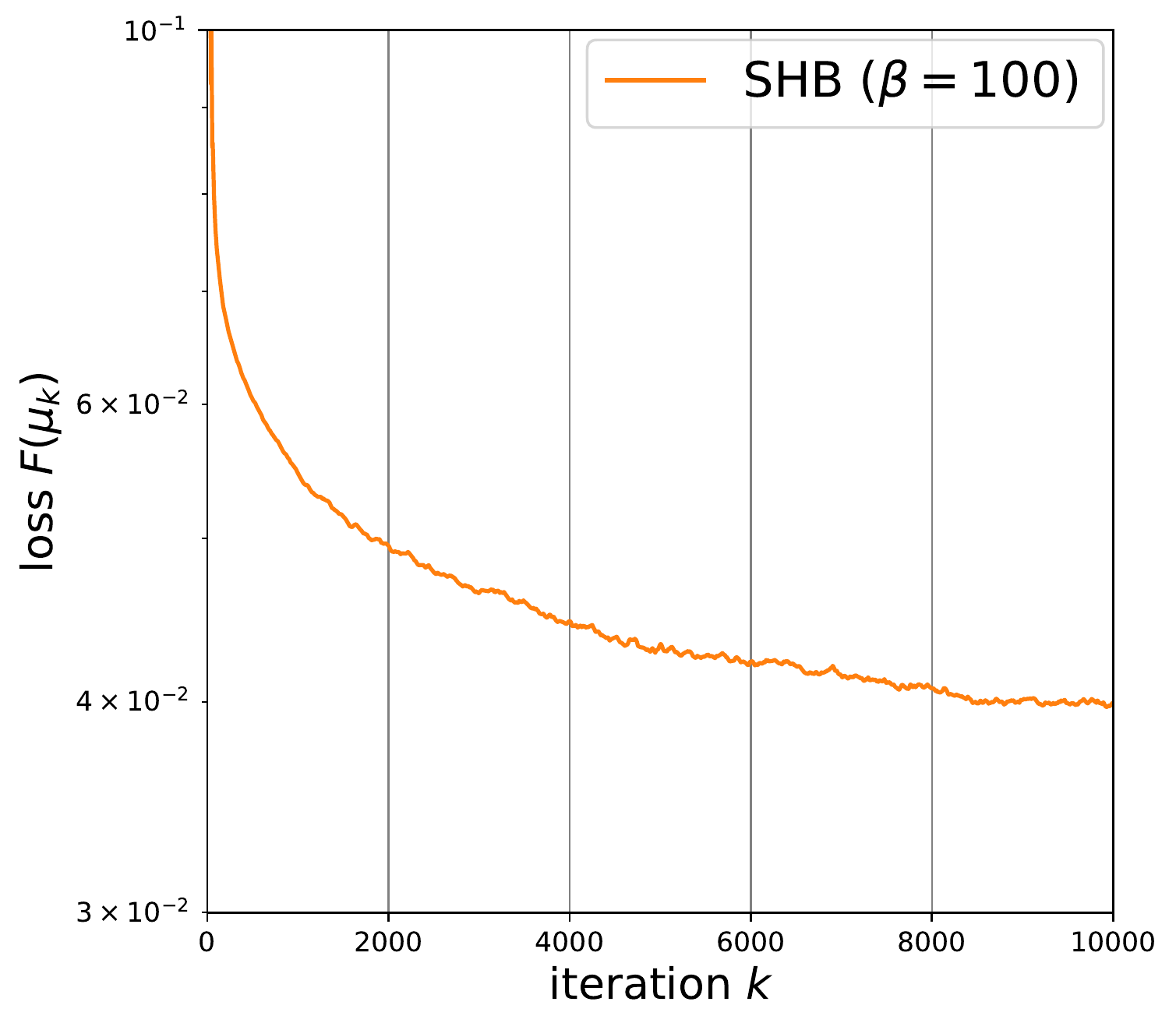}
\caption{Loss $F(\mu^{n}_k)$}
\end{subfigure}%
\begin{subfigure}{0.31\textwidth}
\includegraphics[width=\textwidth]{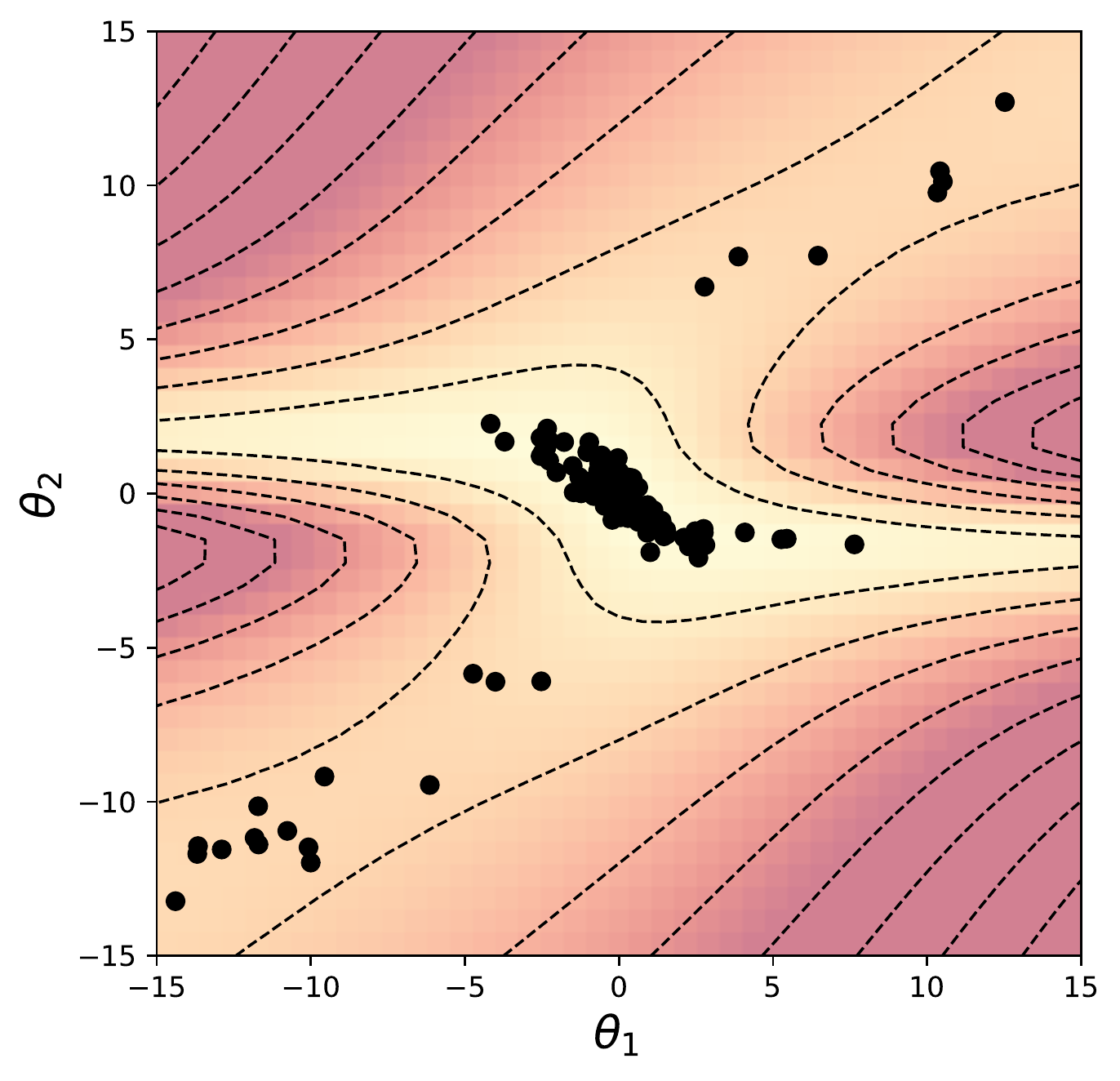}
\caption{$\theta$ marginal}
\label{fig:distributions_theta}
\end{subfigure}%
\begin{subfigure}{0.31\textwidth}
\includegraphics[width=\textwidth]{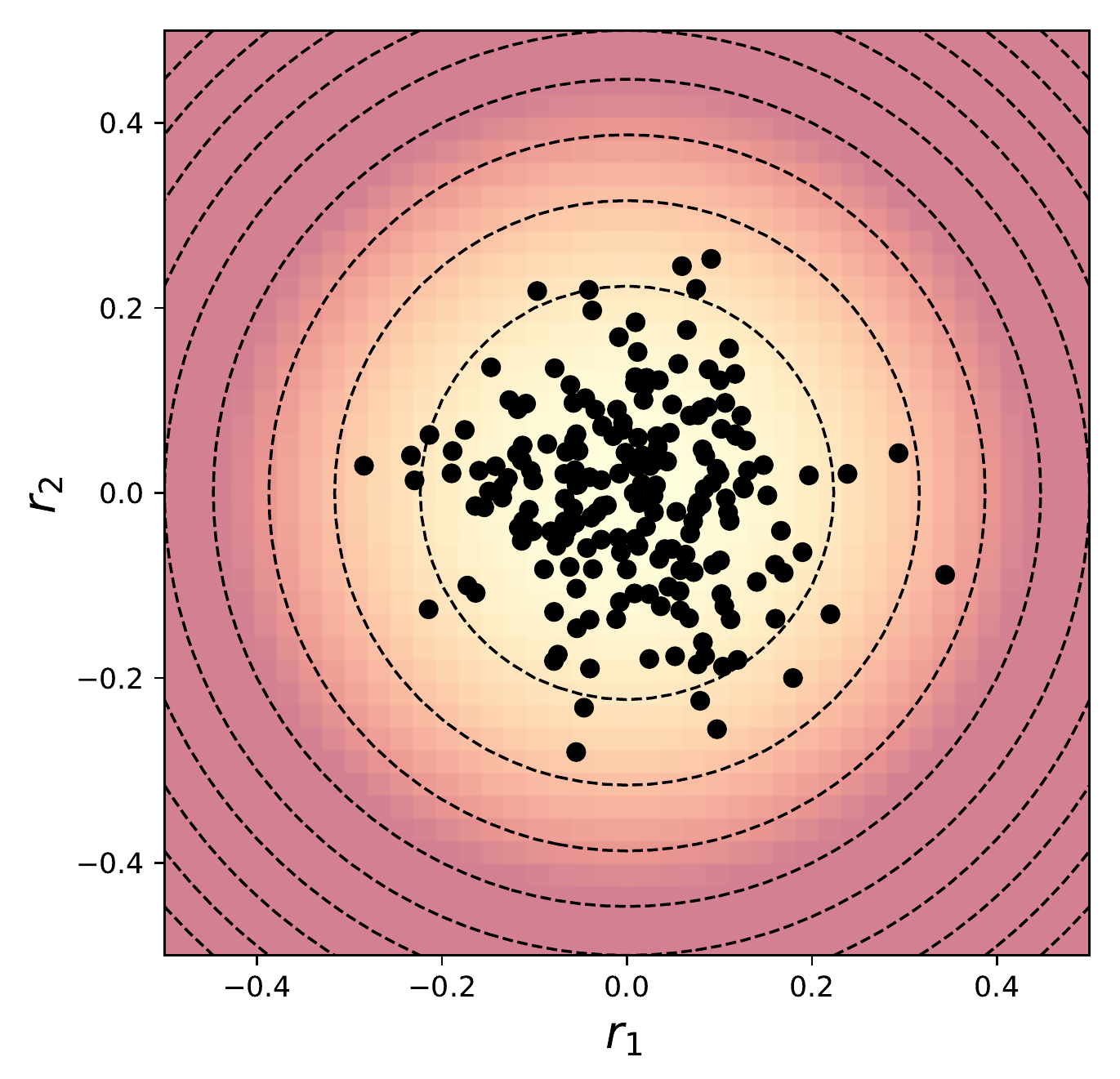}
\caption{$r$ marginal}
\end{subfigure}%
\caption{Illustration of the limiting distribution under stochastic heavy ball dynamics. The loss value as a function of iteration number is shown on the left. The middle and the right plots show the marginal distributions of position $\theta$ and velocity $r$, at the last iteration $k = 10^5$ (scatter plot). The heat map shows a numerical approximation of the theoretical limiting distributions according to Theorem~\ref{thm:decomposition}. The level sets represent the log of the density of the marginals, i.e. $-\beta F'(\mu_k^n)(\theta)$, and $-\beta |r|^2/2$ respectively.}
\label{fig:distributions}
\end{figure}

\subsection{Illustration of the interaction potential}
Finally, we illustrate the interpretation of the learning dynamics as interacting particles. One can view the dynamics of the network parameters $(\theta_i)_{i = 1, \dots, n}$ as evolving in a static potential given by the loss function $f(\theta_1, \dots, \theta_n) = F\left(\frac{1}{n} \sum_{i = 1}^n \delta_{\theta_i}\right)$, defined on $\Theta^n$. But because $\nabla_{\theta_j} f(\theta_1, \dots, \theta_n) = \frac{1}{n} \nabla F'\left(\frac{1}{n} \sum_i \delta_{\theta_i}\right)(\theta_j)$ (see Appendix~\ref{app:gradient}), in fact each of the $n$ particles is subject to the \emph{same time-varying potential} $F'(\mu^n): \Theta \to \Rbb$, where $\mu^n = \frac{1}{n} \sum_{i = 1}^n \delta_{\theta_i}$ is the empirical distribution. The potential at any time depends on the joint distribution of particles at that time, but as the distribution converges, the interaction potential also converges. To illustrate this, we plot in Figure~\ref{fig:interaction_potential} the evolution of $F'(\mu^n_k)$ as the step $k$ increases.

\begin{figure}[h]
\centering
\begin{subfigure}{0.25\textwidth}
\includegraphics[width=\textwidth]{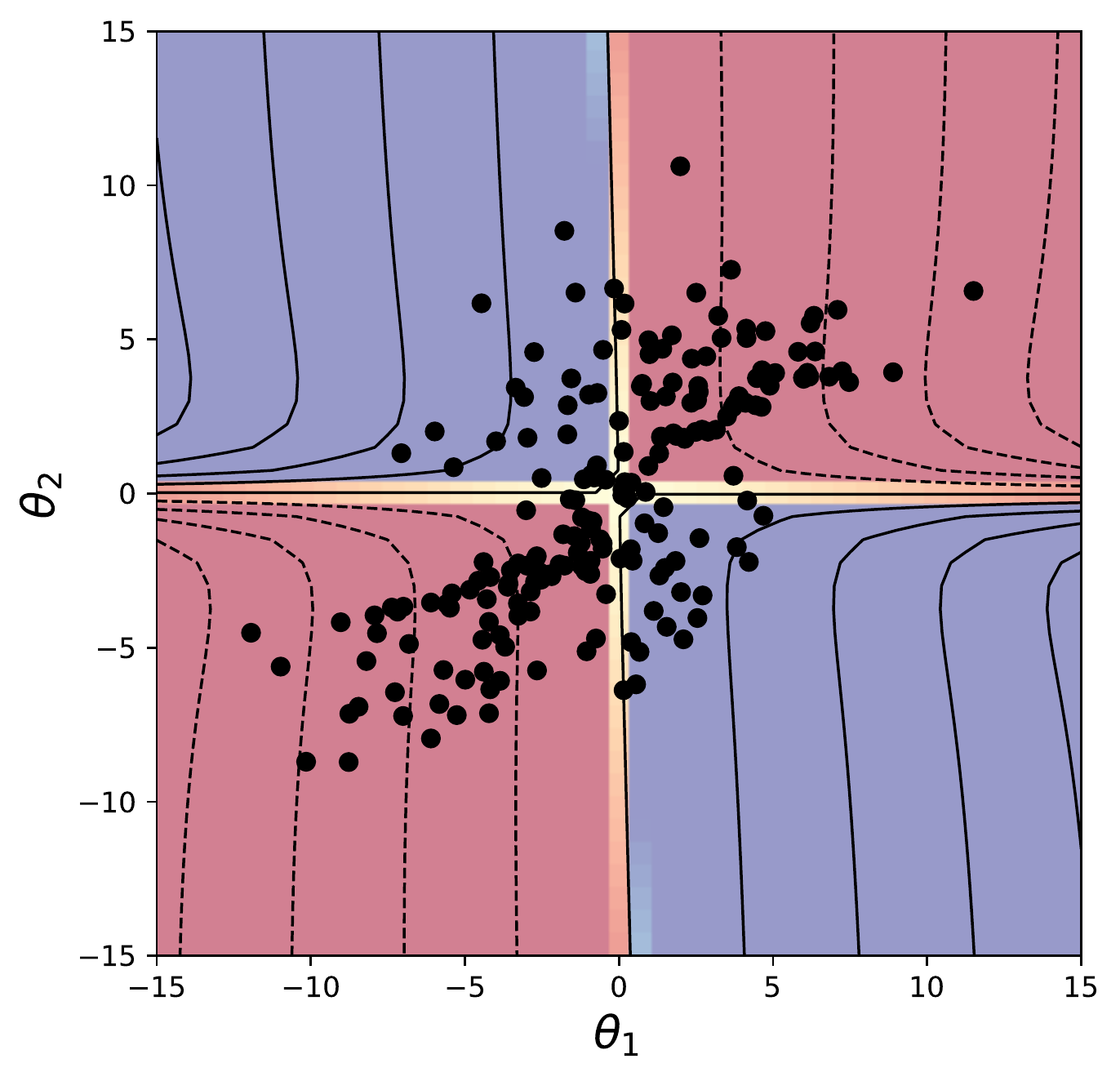}
\caption{$k = 10$}
\end{subfigure}%
\begin{subfigure}{0.25\textwidth}
\includegraphics[width=\textwidth]{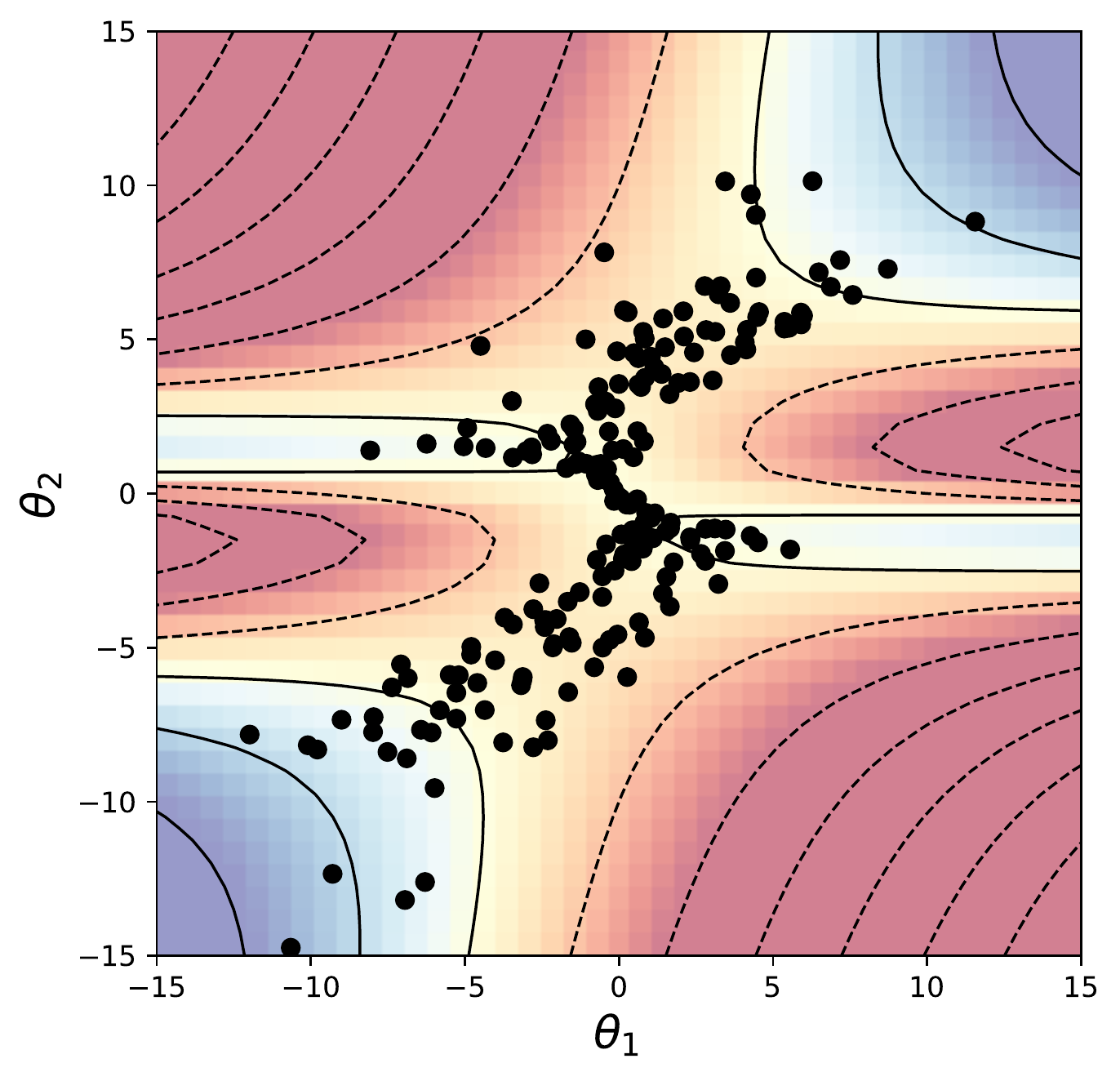}
\caption{$k = 100$}
\end{subfigure}%
\begin{subfigure}{0.25\textwidth}
\includegraphics[width=\textwidth]{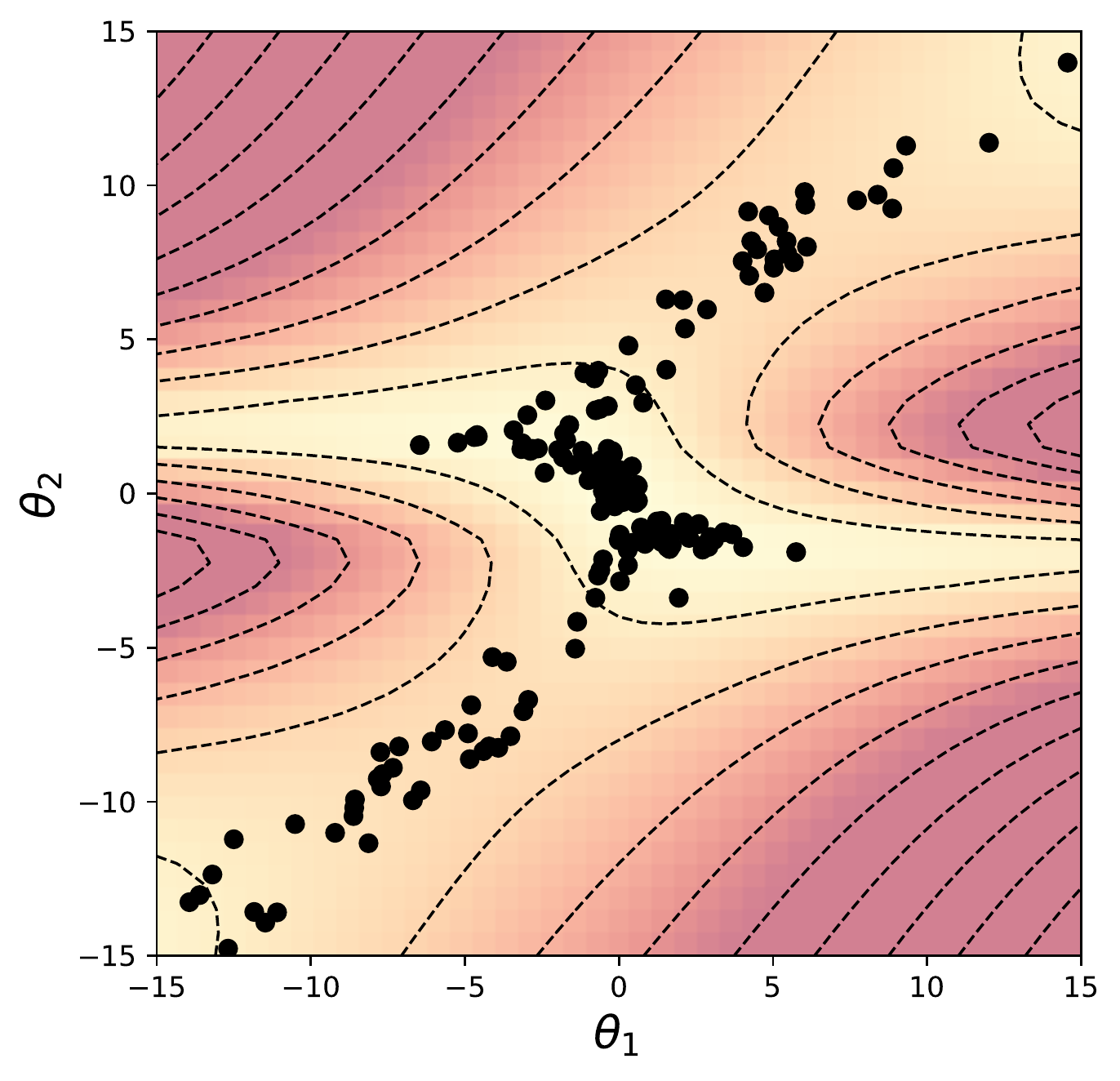}
\caption{$k = 1000$}
\end{subfigure}%
\begin{subfigure}{0.25\textwidth}
\includegraphics[width=\textwidth]{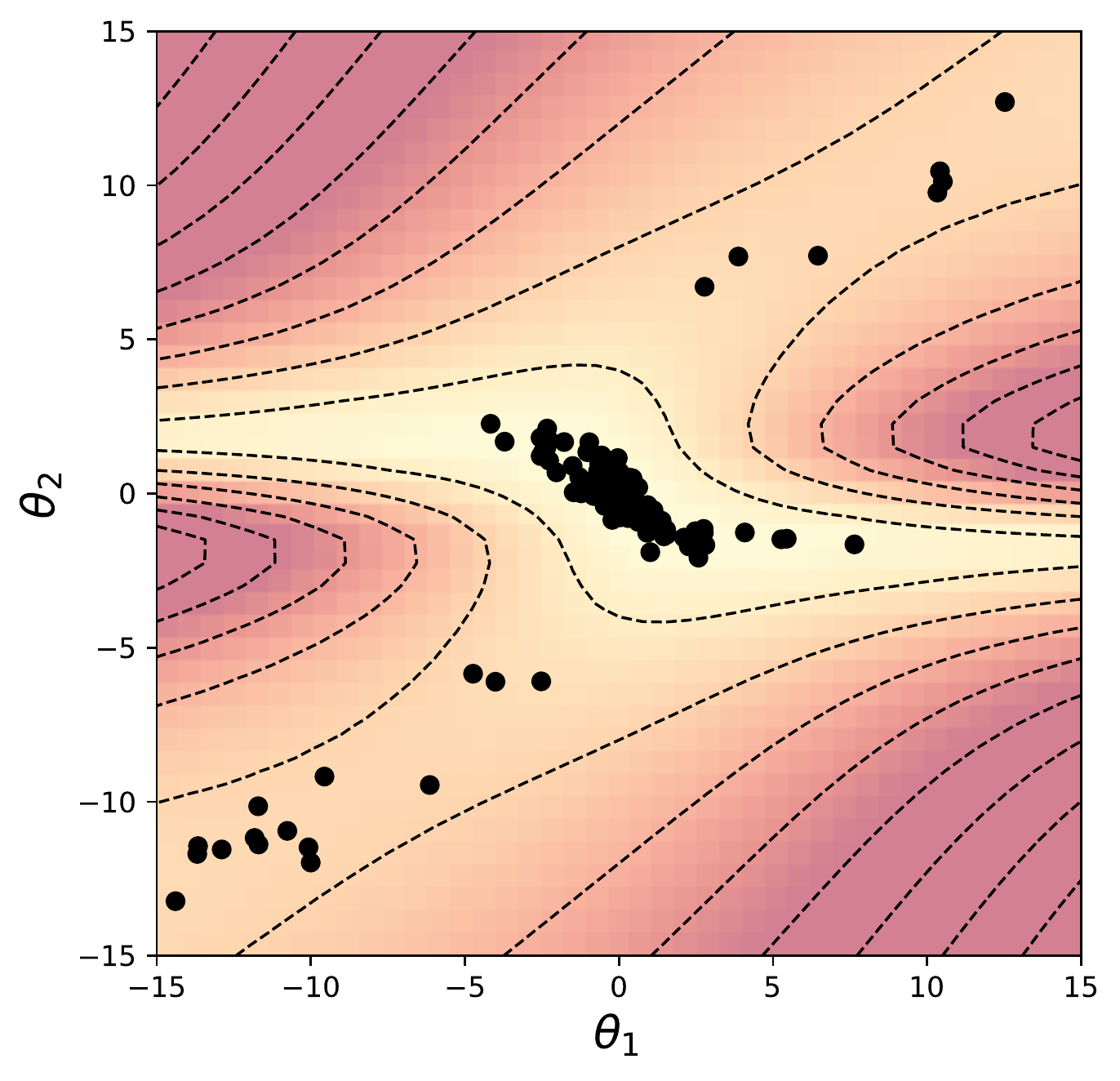}
\caption{$k = 10000$}
\end{subfigure}%
\caption{Evolution of the interaction potential $F'(\mu_k^n)$ as $k$ increases.}
\label{fig:interaction_potential}
\end{figure}

\section{Concluding remarks}
We studied the stochastic heavy ball dynamics in the mean field limit, and established convergence to global minimizers. This is, to our knowledge, the first global convergence guarantee for second-order dynamics in this setting. Though the result is asymptotic, numerical experiments on synthetic problems suggest that the convergence occurs for networks of reasonably small size.

There are several possible directions to investigate quantitative results. For example, hypocoercivity~\cite{villani2009hypocoercivity} is concerned with the study of the rate of convergence of $\rho_t$ to its limiting distribution, and while the theory is in its early development for the nonlinear case, we believe the techniques can be adapted under additional assumptions on $F$. A second direction is the study of fluctuations of solutions, which quantifies the convergence of $\mu_t^n$ to the mean field limit $\mu_t$ as the number of particles $n \to \infty$, as was done in~\cite{sirignano2018mean,rotskoff2018parameters} for gradient flow with quadratic loss. A third direction is to study the generalization properties of the limit. In the gradient flow case, this was investigated for instance by~\cite{chizat2020implicit} for the logistic loss.

We believe our results can be generalized to a broader family of second-order dynamics, including Nesterov's method. One technical challenge in doing so is that the dynamics has a time-dependence due to the damping coefficient $\gamma_t$, which may require using a different Lyapunov functional.

Finally, the question of convergence for the noiseless second-order dynamics remains unsettled and requires further investigation. In general, without diffusion, there may exist stationary points that are not global minimizers (even in the mean field limit). However, one can hope to prove, under suitable assumptions on $F$, that such stationary points are repulsive, as was done in~\cite{chizat2018global} for gradient flow.

\newpage
\small
\bibliographystyle{plain}
\bibliography{second_order}

\begin{thebibliography}{10}

\bibitem{abdallah1995child}
Naoufel~Ben Abdallah and Pierre Degond.
\newblock The {C}hild-{L}angmuir law in the kinetic theory of charged
  particles: Semiconductors models.
\newblock {\em Mathematical Problems in Semiconductor Physics}, 340:76, 1995.

\bibitem{albiac2006topics}
Fernando Albiac and Nigel~John Kalton.
\newblock {\em Topics in Banach space theory}, volume 233.
\newblock Springer, 2006.

\bibitem{ambrosio2008gradient}
Luigi Ambrosio, Nicola Gigli, and Giuseppe Savar{\'e}.
\newblock {\em Gradient flows: in metric spaces and in the space of probability
  measures}.
\newblock Springer Science \& Business Media, 2008.

\bibitem{attouch2000heavy}
Hedy Attouch and Felipe Alvarez.
\newblock The heavy ball with friction dynamical system for convex constrained
  minimization problems.
\newblock In {\em Lecture Notes in Econom. and Math. Systems}, pages 25--35.
  Springer, 2000.

\bibitem{bach2017breaking}
Francis Bach.
\newblock Breaking the curse of dimensionality with convex neural networks.
\newblock {\em The Journal of Machine Learning Research}, 18(1):629--681, 2017.

\bibitem{bakry2008rate}
Dominique Bakry, Patrick Cattiaux, and Arnaud Guillin.
\newblock Rate of convergence for ergodic continuous markov processes: Lyapunov
  versus poincar{\'e}.
\newblock {\em J. Funct. Anal.}, 254(3):727--759, 2008.

\bibitem{barron1993universal}
Andrew~R Barron.
\newblock Universal approximation bounds for superpositions of a sigmoidal
  function.
\newblock {\em IEEE Transactions on Information theory}, 39(3):930--945, 1993.

\bibitem{batt1993rigorous}
J{\"u}rgen Batt and Gerhard Rein.
\newblock A rigorous stability result for the {Vlasov-Poisson} system in three
  dimensions.
\newblock {\em Annali di matematica pura ed applicata}, 164(1):133--154, 1993.

\bibitem{bengio2006convex}
Yoshua Bengio, Nicolas Le~Roux, Pascal Vincent, Olivier Delalleau, and Patrice
  Marcotte.
\newblock Convex neural networks.
\newblock In {\em Advances in neural information processing systems}, pages
  123--130, 2006.

\bibitem{billingsley2013convergence}
Patrick Billingsley.
\newblock {\em Convergence of probability measures}.
\newblock John Wiley \& Sons, 2013.

\bibitem{bouchut1995long}
Fran{\c{c}}ois Bouchut and Jean Dolbeault.
\newblock On long time asymptotics of the {V}lasov-{F}okker-{P}lanck equation
  and of the {Vlasov-Poisson-Fokker-Planck system with Coulombic and Newtonian}
  potentials.
\newblock {\em Differential and Integral Equations}, 8(3):487--514, 1995.

\bibitem{cabot2009long}
Alexandre Cabot, Hans Engler, and S{\'e}bastien Gadat.
\newblock On the long time behavior of second order differential equations with
  asymptotically small dissipation.
\newblock {\em Transactions of the American Mathematical Society},
  361(11):5983--6017, 2009.

\bibitem{carrillo2003kinetic}
Jos{\'e}~A Carrillo, Robert~J McCann, and C{\'e}dric Villani.
\newblock Kinetic equilibration rates for granular media and related equations:
  entropy dissipation and mass transportation estimates.
\newblock {\em Revista Matematica Iberoamericana}, 19(3):971--1018, 2003.

\bibitem{chandrasekhar1943stochastic}
Subrahmanyan Chandrasekhar.
\newblock Stochastic problems in physics and astronomy.
\newblock {\em Reviews of modern physics}, 15(1):1, 1943.

\bibitem{chaumont2012exercises}
Lo{\"\i}c Chaumont and Marc Yor.
\newblock {\em Exercises in Probability: a guided tour from measure theory to
  random processes, via conditioning}.
\newblock Number~35 in Cambridge Series in Statistical and Probabilistic
  Mathematics. Cambridge University Press, 2012.

\bibitem{chizat2018global}
Lena\"ic Chizat and Francis Bach.
\newblock On the global convergence of gradient descent for over-parameterized
  models using optimal transport.
\newblock In {\em Advances in neural information processing systems}, pages
  3036--3046, 2018.

\bibitem{chizat2020implicit}
Lena\"ic Chizat and Francis Bach.
\newblock Implicit bias of gradient descent for wide two-layer neural networks
  trained with the logistic loss.
\newblock {\em arXiv preprint arXiv:2002.04486}, 2020.

\bibitem{chizat2019lazy}
Lena\"ic Chizat, Edouard Oyallon, and Francis Bach.
\newblock On lazy training in differentiable programming.
\newblock In {\em Advances in Neural Information Processing Systems}, pages
  2933--2943, 2019.

\bibitem{cybenko1989approximation}
George Cybenko.
\newblock Approximation by superpositions of a sigmoidal function.
\newblock {\em Mathematics of control, signals and systems}, 2(4):303--314,
  1989.

\bibitem{degond2000mathematical}
Pierre Degond.
\newblock Mathematical modelling of microelectronics semiconductor devices.
\newblock {\em AMS/IP Studies in Advanced Mathematics}, 15:77--110, 2000.

\bibitem{dolbeault1991stationary}
Jean Dolbeault.
\newblock Stationary states in plasma physics: Maxwellian solutions of the
  {Vlasov-Poisson} system.
\newblock {\em Mathematical Models and Methods in Applied Sciences},
  1(2):183--208, 1991.

\bibitem{gadat2018stochastic}
S{\'e}bastien Gadat, Fabien Panloup, and Sofiane Saadane.
\newblock Stochastic heavy ball.
\newblock {\em Electronic Journal of Statistics}, 12(1):461--529, 2018.

\bibitem{gilbarg2015elliptic}
David Gilbarg and Neil~S Trudinger.
\newblock {\em Elliptic partial differential equations of second order}.
\newblock Springer, 2015.

\bibitem{goudou2009gradient}
Xavier Goudou and Julien Munier.
\newblock The gradient and heavy ball with friction dynamical systems: the
  quasiconvex case.
\newblock {\em Mathematical Programming}, 116(1-2):173--191, 2009.

\bibitem{haraux1991systemes}
Alain Haraux.
\newblock {\em Syst\`emes dynamiques dissipatifs et applications}, volume~17.
\newblock Masson, 1991.

\bibitem{huang2000variational}
Chaocheng Huang and Richard Jordan.
\newblock Variational formulations for {Vlasov--Poisson--Fokker--Planck}
  systems.
\newblock {\em Mathematical methods in the applied sciences}, 23(9):803--843,
  2000.

\bibitem{jacot2018neural}
Arthur Jacot, Franck Gabriel, and Cl{\'e}ment Hongler.
\newblock Neural tangent kernel: Convergence and generalization in neural
  networks.
\newblock In {\em Advances in neural information processing systems}, pages
  8571--8580, 2018.

\bibitem{jordan1998variational}
Richard Jordan, David Kinderlehrer, and Felix Otto.
\newblock The variational formulation of the {F}okker--{P}lanck equation.
\newblock {\em SIAM journal on mathematical analysis}, 29(1):1--17, 1998.

\bibitem{kingma2014adam}
Diederik~P Kingma and Jimmy Ba.
\newblock Adam: A method for stochastic optimization.
\newblock In {\em International Conference on Learning Representations}, 2015.

\bibitem{mei2018mean}
Song Mei, Andrea Montanari, and Phan-Minh Nguyen.
\newblock A mean field view of the landscape of two-layer neural networks.
\newblock {\em Proceedings of the National Academy of Sciences},
  115(33):E7665--E7671, 2018.

\bibitem{nesterov1983method}
Yurii Nesterov.
\newblock A method of solving a convex programming problem with convergence
  rate $o (1/k^2)$.
\newblock In {\em Sov. Math. Dokl}, volume~27, pages 372--376, 1983.

\bibitem{oelschlager1984martingale}
Karl Oelschl{\"a}ger.
\newblock A martingale approach to the law of large numbers for weakly
  interacting stochastic processes.
\newblock {\em The Annals of Probability}, pages 458--479, 1984.

\bibitem{oelschlager1985law}
Karl Oelschl{\"a}ger.
\newblock A law of large numbers for moderately interacting diffusion
  processes.
\newblock {\em Zeitschrift f{\"u}r Wahrscheinlichkeitstheorie und verwandte
  Gebiete}, 69(2):279--322, 1985.

\bibitem{padmanabhan1990statistical}
Thanu Padmanabhan.
\newblock Statistical mechanics of gravitating systems.
\newblock {\em Physics Reports}, 188(5):285--362, 1990.

\bibitem{pavliotis2014stochastic}
Grigorios~A Pavliotis.
\newblock {\em Stochastic processes and applications: diffusion processes, the
  {Fokker-Planck and Langevin} equations}, volume~60.
\newblock Springer, 2014.

\bibitem{perthame2004mathematical}
Beno\^{i}t Perthame.
\newblock Mathematical tools for kinetic equations.
\newblock {\em Bulletin of the American Mathematical Society}, 41(2):205--244,
  2004.

\bibitem{polyak1964some}
Boris~T Polyak.
\newblock Some methods of speeding up the convergence of iteration methods.
\newblock {\em USSR Computational Mathematics and Mathematical Physics},
  4(5):1--17, 1964.

\bibitem{rotskoff2018parameters}
Grant Rotskoff and Eric Vanden-Eijnden.
\newblock Parameters as interacting particles: long time convergence and
  asymptotic error scaling of neural networks.
\newblock In {\em Advances in neural information processing systems}, pages
  7146--7155, 2018.

\bibitem{sirignano2018mean}
Justin Sirignano and Konstantinos Spiliopoulos.
\newblock Mean field analysis of neural networks: A law of large numbers.
\newblock {\em arXiv preprint arXiv:1805.01053}, 2018.

\bibitem{slotine1991applied}
Jean-Jacques~E Slotine and Weiping Li.
\newblock {\em Applied nonlinear control}, volume 199.
\newblock Prentice hall Englewood Cliffs, NJ, 1991.

\bibitem{soler1997asymptotic}
Juan Soler, Jos{\'e}~A Carrillo, and Luis~L Bonilla.
\newblock Asymptotic behavior of an initial-boundary value problem for the
  {Vlasov--Poisson--Fokker--Planck} system.
\newblock {\em SIAM Journal on Applied Mathematics}, 57(5):1343--1372, 1997.

\bibitem{su2014differential}
Weijie Su, Stephen Boyd, and Emmanuel Cand\`es.
\newblock A differential equation for modeling {N}esterov’s accelerated
  gradient method: Theory and insights.
\newblock In {\em Advances in Neural Information Processing Systems}, pages
  2510--2518, 2014.

\bibitem{sutskever2013importance}
Ilya Sutskever, James Martens, George Dahl, and Geoffery Hinton.
\newblock On the importance of momentum and initialization in deep learning.
\newblock In {\em 30th International Conference on Machine Learning}, pages
  404--439, 2013.

\bibitem{victory1990classical}
Harold~Dean Victory~Jr and Brian~P O'Dwyer.
\newblock On classical solutions of {Vlasov-Poisson Fokker-Planck} systems.
\newblock {\em Indiana University mathematics journal}, pages 105--156, 1990.

\bibitem{villani2009hypocoercivity}
C{\'e}dric Villani.
\newblock Hypocoercivity.
\newblock {\em Memoirs of the American Mathematical Society}, 202(950), 2009.

\bibitem{wibisono2016variational}
Andre Wibisono, Ashia~C Wilson, and Michael~I Jordan.
\newblock A variational perspective on accelerated methods in optimization.
\newblock {\em Proceedings of the National Academy of Sciences},
  113(47):E7351--E7358, 2016.

\bibitem{williams1991probability}
David Williams.
\newblock {\em Probability with martingales}.
\newblock Cambridge university press, 1991.

\bibitem{woodworth2019kernel}
Blake Woodworth, Suriya Gunasekar, Jason Lee, Daniel Soudry, and Nathan Srebro.
\newblock Kernel and rich regimes in overparametrized models.
\newblock {\em arXiv preprint arXiv:1906.05827}, 2019.

\end{thebibliography}

\newpage

\appendix
\normalsize
\section{Summary of notations}
\begin{table}[h]
\centering
\renewcommand{\arraystretch}{1.2}
\begin{tabular}{ l|p{10cm} }
$\mathcal{T}\Theta$ & The tangent bundle of $\Theta$\\ 
$\Mcal(\mathcal{T}\Theta)$ & Set of probability distributions over $\Tcal\Theta$ \\
$\Mac(\mathcal{T}\Theta)$ & Set of probability distributions over $\Tcal\Theta$ that are absolutely continuous w.r.t. the Lebesgue measure $\differential\theta\differential r$, where $(\theta,r)\in\Tcal\Theta$ \\
$\Pcal(\mathcal{T}\Theta)$ & Set of joint probability density functions (PDFs) over $\Tcal\Theta$ \\
\hline
$\Fcal$ & Hilbert space of functions from $\mathbb{R}^{d}$ to $\mathbb{R}$\\
$\braketf{\phi}{\psi}$ & $\mathbb E_{(x,y)} [\phi(x, y)\psi(x, y)]$ for $\phi, \psi \in \Fcal$ \\ 
$\|\phi\|_\Fcal$ & Hilbert norm of $\phi \in \Fcal$ \\
\hline
$\dotp{\theta}{\theta'}$ & Dot product of two vectors $\theta, \theta' \in \Rbb^d$ \\
$|\theta|$ & Euclidean norm of a vector $\theta \in \Rbb^d$ \\
\hline
$\dotp{f}{g}$ & $\int_{\Tcal\Theta} f(\theta,r) g(\theta,r) \differential\theta\differential r$, for $f, g : \Tcal\Theta \mapsto \Rbb$ \\ 
$\braketd{u}{v}$ & $\int_{\Tcal\Theta} \dotp{u(\theta,r)}{v(\theta,r)} \differential\theta\differential r$, for two vector fields $u,v : \Tcal\Theta \mapsto \Rbb^d$ \\
$\|f\|_p$ & $L^p$ norm of $f : \Tcal\Theta \mapsto \Rbb$ \\
$[\rho]^\theta$ & $\theta$-marginal density function of the joint density function $\rho(\theta,r)$\\
\hline
$\nabla$ & Gradient\\
$\hess\left(\cdot\right)$ & Hessian\\
$\Delta$ & Laplacian 
\end{tabular}
\renewcommand{\arraystretch}{1}
\end{table}

We sometimes put a subscript to the gradient, Hessian and Laplacian operators, to clarify these operators are acting w.r.t. which variable. We often omit the subscripts if the operator is taken w.r.t. the full vector, and not w.r.t. a sub-vector.

\section{Gradient in parameter space and differential in distribution space}
\label{app:gradient}
Recall that the risk functional is $R : \Fcal \to \Rbb_+$. Let us define the following functions:
\begin{itemize}
\item The objective function in parameter space:
\[
f(\theta_1, \dots, \theta_n) := F\left(\frac{1}{n}\sum_{i = 1}^n \delta_{\theta_i}\right) = R\left(\frac{1}{n} \sum_{i=1}^n \Psi(\theta_i)\right) + \frac{1}{n}\sum_{i = 1}^n g(\theta_i).
\]
\item The objective functional in distribution space: $F(\mu) := R(\braket{\Psi}{\mu}) + \braket{g}{\mu}$.
\end{itemize}
By the chain rule, the gradient of $f$ is
\begin{equation}
\label{eq:gradient}
\nabla_{\theta_j} f(\theta_1, \dots, \theta_n) = \frac{1}{n}\braketf{R'(\textstyle{\frac{1}{n} \sum_{i = 1}^n\Psi(\theta_i)})}{\Psi'(\theta_j)} + \frac{1}{n}\nabla g(\theta_j),
\end{equation}
and the differential of $F$ is
\begin{equation}
\label{eq:dF}
F'(\mu)(\cdot) = \braketf{R'(\braket{\Psi}{\mu})}{\Psi(\cdot)} + g(\cdot).
\end{equation}

Identifying~\eqref{eq:gradient} and~\eqref{eq:dF}, we see that $n\nabla_{\theta_j} f(\theta_1, \dots, \theta_n) = \nabla F'(\mu^n)(\theta_j)$, where $\mu^n = \frac{1}{n} \sum_{i=1}^n \delta_{\theta_i}$.

Observe that the gradient is scaled by $n$, which leads to nonlinear dynamics in the mean field limit. A different scaling can lead to simpler, linearized dynamics that are referred to as lazy training~\citen{{chizat2019lazy}} or the kernel regime~\citen{{jacot2018neural,woodworth2019kernel}}. Our analysis is concerned with the fully non-linear regime.

The stochastic heavy ball dynamics in the parameter space is given by
\[
\begin{cases}
\dot \theta_i = r_i,\\
\dot r_i = -n\nabla_{\theta_i} f(\theta_1, \dots, \theta_n) - \gamma r_i + \sqrt{2\gamma\beta^{-1}}\:\differential W_t^{i},
\end{cases}
\]
where $i=1,\hdots,n$. Using $n\nabla_{\theta_i} f(\theta_1, \dots, \theta_n) = \nabla F'(\mu^n)$ in the previous equation yields~\eqref{eq:second_order_particle}.

\section{Consistency of the mean field limit}
\label{app:consistency}
\begin{proof}[Proof of Theorem \ref{thm:consistency}]
The proof of consistency follows a standard martingale argument, which we briefly sketch here. Additional details can be found in~\citen{{oelschlager1984martingale,oelschlager1985law}}.

For $i=1,\hdots,n$, let
\begin{align*}
\xtni:= \begin{pmatrix} \theta^i_t \\ r^i_t 
\end{pmatrix}, \quad  b(\xtni,\mutn):=\begin{pmatrix} r^i_t \\    F'([\mu_{t}^{n}]^{\theta^i_t})(\theta^i_t)- \gamma r^i_t 
\end{pmatrix},\quad 
\sigma(\xtni,\mutn):=\sqrt{2\beta^{-1}\gamma} \begin{pmatrix}
0_{d \times d} \\ I_{d\times d} \end{pmatrix},
\end{align*}
and consider the system of It\^{o} stochastic differential equations:
\begin{align}
\diff\xtni & = 
b(\xtni,\mutn) \: \diff t
+ \sigma(\xtni,\mutn) \:\diff W^i_t,
\end{align}
where $\diff W^i_t$, for each $i=1,\hdots,n$, is the standard Wiener process in $\Tcal\Theta$.

For any compactly supported test function $\varphi\in C^2_{b}(\mathcal{T}\Theta)$, i.e., the space of all bounded continuous functions $\varphi : \mathcal{T}\Theta \mapsto \mathbb{R}$ with bounded continuous partial derivatives of first and second order, we want to describe the time evolution of the quantity
\begin{align}
    \langle \varphi,\mu_{t}^n \rangle= \frac{1}{n} \sum_{i=1}^{n} \varphi(X_t^{i}).
\end{align}
Using It\^o's rule, we have 
\begin{align*}
     \diff \varphi(\xtni) = 
     L_{\mutn}\varphi(\xtni)\: \diff t+ \nabla \varphi^{\top}(\xtni)  \sigma(\xtni,\mutn) \: \diff W_t^i,
\end{align*} 
wherein the infinitesimal generator $L$ is defined as
\begin{align}
    L_\mu\varphi(x) := \langle b(x,\mu), \nabla_x \varphi(x) \rangle + \frac{1}{2}\tr\left(\sigma\sigma^{\top}(x,\mu){\rm{Hess}}(\varphi) \right).
\label{L}    
\end{align}
Therefore,
\begin{align} \label{PDEfinite}
    \diff \langle \varphi,\mutn \rangle &= \frac{1}{n} \sum_{i=1}^{n}\diff \varphi(\xtni) \nonumber \\ 
    & = \langle L_{\mutn}\varphi,\mutn   \rangle \:\diff t + \frac{1}{n} \sum_{i=1}^{n} \nabla \varphi^{\top}(\xtni) \sigma(\xtni,\mutn) \: \diff W_t^i \nonumber \\ 
    & := \langle L_{\mutn}\varphi,\mutn   \rangle \:\diff t + \diff M^n_t,
\end{align}
where $M^n_t$ is a local martingale. Since $\varphi \in C^{2}_b(\mathcal{T}\Theta)$, we have  $\lvert\nabla \varphi^{\top} \sigma \lvert \leq  \sqrt{2\beta^{-1}\gamma}\lvert \nabla_r \varphi  \rvert \leq C$ uniformly for some $C>0$. Notice that the quadratic variation of the noise term in (\ref{PDEfinite}) is
\begin{align*}
    [M^n_t]= \frac{1}{n^2} \sum_{i=1}^{n}\int_{0}^t \lvert \nabla \varphi^{\top}(X_s^{i}) \sigma(X_s^{ i},\mu_{s}^n)\rvert^2  \: \diff s  \leq \frac{tC^2}{n},
\end{align*}
and by Doob's martingale inequality, we deduce that 
\begin{align}
\mathbb{E}\left (\sup_{t\leq T}  M^n_t \right)^2 \leq \mathbb{E}\left( \sup_{t\leq T}  (M^n_t ) ^2\right) \leq 4 \mathbb{E}  \left( (M^n_t)^2 \right)  \leq 4  \mathbb{E} ([M^n_t])  \leq   \frac{4tC^2}{n}.  
\end{align}
So as $n \rightarrow \infty$, the noise term in (\ref{PDEfinite}) converges to zero in probability, and we get a deterministic evolution equation.

Next, we argue that sequence $\left\{ (\mutn)_{t>0}\right\}_{n=1}^{\infty}$ of measure-valued stochastic processes converges to some probability measure-valued limiting process $(\mu_t)_{t>0}$ as $n \rightarrow \infty$. To this end, we take $\left\{ (\mutn)_{t>0}\right\}_{n=1}^{\infty}$ to be the (random) elements of $\Omega=C([0,\infty),\mathcal{M}(\mathcal{T}\Theta))$, the set of continuous functions from $[0,\infty)$ into $\mathcal{M}(\mathcal{T}\Theta)$ endowed with the topology of weak convergence.
Following~\citen{{oelschlager1984martingale,oelschlager1985law}}, it can be shown that the sequence $\mathbb{P}_n$ of probability measures on $\Omega$ induced by the processes $\left\{ (\mutn)_{t>0}\right\}_{n=1}^{\infty}$ weakly converges (along a subsequence) to some $\mathbb{P}$, where $\mathbb{P}$ is the measure induced by the limiting process $(\mu_t)_{t>0}$. By Skorohod's representation theorem \citen{[Theorem 6.7]{billingsley2013convergence}}, the sequence $\left\{ (\mutn)_{t>0}\right\}_{n=1}^{\infty}$ converges $\mathbb{P}$-almost surely to $(\mu_t)_{t>0}$. Since the martingale term in (\ref{PDEfinite}) vanishes as $n \rightarrow \infty$, we obtain
\begin{align}
    \diff \langle \varphi,\mu_t \rangle = \langle L_{\mu}\varphi,\mu_t \rangle \: \diff t = \langle \varphi,L^{*}_{\mu}\mu_t \rangle \: \diff t,
    \label{Lstar}
\end{align}
which is valid almost everywhere for any test function $\varphi \in C^{2}_b(\mathcal{T}\Theta)$. In (\ref{Lstar}), $L^{*}$ is the adjoint operator of $L$ given by (\ref{L}), and is defined as 
\begin{align}
    L^{*}_{m}\mu(x) := -\nabla \cdot(\mu b(x,m)) + \frac{1}{2} \sum_{i,j=1}^{n} \frac{\partial^2}{\partial x_ix_j}(\mu \sigma\sigma^{\top}(x,m))_{ij}.
\end{align}
This shows that $\mu_t$ is almost surely a weak solution to the nonlinear Fokker-Planck PDE~\eqref{eq:PDE_second_order}.
\end{proof}

\section{Variations and bounds on the free energy}
\label{app:lyap}

This section provides the details of the proofs in Section~\ref{sec:LyapunovFunctional}, and additional bounds that are used in the proofs of the main results.

\subsection{Proof of Lemma~\ref{lem:variation}}
\begin{proof}
By assumption, $\rho_t$ satisfies the continuity equation $\partial \rho_t = -\nabla \cdot (\rho_t v_t)$, thus
\begin{align*}
\partial_t \Vcal(\rho_t)
&= \dotp{\Vcal'(\rho_t)}{\partial_t \rho_t} & \text{by the chain rule}\\
&= \dotp{\Vcal'(\rho_t)}{-\nabla \cdot (\rho_t v_t)} &\text{by the continuity equation}\\
&= \braketd{\nabla \Vcal'(\rho_t)}{\rho_t v_t},
\end{align*}
where the last equality follows by duality of the gradient and divergence operators $\nabla$ and $\nabla \cdot$, in the following sense: if $f : \Tcal\Theta \to \Rbb$ is a differentiable scalar function and $G: \Tcal\Theta \to \Rbb^d$ is a vector field, then
\begin{equation}
\label{eq:duality}
\braketd{\nabla f}{G} + \braket{f}{\nabla \cdot G} = 0.
\end{equation}
The inner product in the first summand above is for vector fields whereas the same in the second summand is for scalar-valued functions.
\end{proof}

\subsection{Lyapunov function in the single particle case}
This section highlights a connection between Lyapunov functions for the single particle case, and Lyapunov functionals for the mean-field dynamics. To simplify the notation, let $\xi = (\theta, r)$ denote a position-velocity pair. Lemma~\ref{lem:variation} states that if $\rho_t \in \Pcal(\Tcal\Theta)$ solves the continuity equation $\partial_t \rho_t = -\nabla \cdot (\rho_t v(\rho_t))$, then the time derivative of a functional $\Vcal(\rho)$ along the solution trajectory $\rho_t$ is given by
\begin{equation}
\label{eq:variation_mean_field}
\frac{\diff}{\diff t} \Vcal(\rho_t) = \braket{\nabla \Vcal'(\rho_t)}{\rho_t v(\rho_t)}_* = \int_{\Tcal\Theta}\braket{\nabla \Vcal'(\rho_t)(\xi)}{v(\rho_t)(\xi)} \rho_t(\xi) \differential \xi.
\end{equation}
In the single particle case, if $\xi_t \in \Tcal\Theta$ solves the differential equation $\dot \xi_t = v(\xi_t)$ for a vector field $v$, then the time derivative of a function $V(\xi)$ along $\xi_t$ is, by the chain rule,
\begin{equation}
\label{eq:variation_particle}
\frac{\diff}{\diff t} V(\xi_t) = \braket{\nabla V(\xi_t)}{v(\xi_t)}.
\end{equation}
Comparing the two expressions, we see that~\eqref{eq:variation_mean_field} can be viewed as an integral version of~\eqref{eq:variation_particle}. This connection is particularly simple in the linear case with no interaction: suppose $\Vcal(\rho) = \braket{V}{\rho}$ for a differentiable function $V: \Tcal\Theta \to \Rbb$, and the vector field $v$ does not depend on $\rho$. Then the following holds:

\emph{If $\braket{\nabla V(\xi)}{v(\xi)} \leq 0$ for all $\xi$, then $\Vcal$ is non-increasing along $\rho_t$, and $V$ is non-increasing along~$\xi_t$.}

In other words, the same sufficient condition describes Lyapunov functions for $\xi_t$ and Lyapunov functionals for $\rho_t$. In the nonlinear case, the condition becomes:

\emph{If $\braket{\nabla \Vcal'(\rho)(\xi)}{v(\rho)(\xi)} \leq 0$ for all $\rho$ and all $\xi$, then $\Vcal$ is non-increasing along $\rho_t$, and $\Vcal'(\rho)$ is non-increasing along the solution to $\dot \xi_t = v(\rho)(\xi_t)$.}

In this case, the condition describes a \emph{family} of single-particle dynamics $v(\rho)$ and corresponding Lyapunov functions $\Vcal'(\rho)$, where the family is indexed by $\rho$.

We examine the case of the noiseless heavy ball dynamics as an example. In this case, we have
\begin{align}
v_{\text{HB}}(\rho)(\theta, r) &= \vect{r \\ -\nabla F'(\rho)(\theta) - \gamma r}, \label{eq:hb_v}\\
\Vcal_{\text{HB}}(\rho)(\theta, r) &= F(\rho) + \braket{\frac{1}{2}|r|^2}{\rho}, \label{eq:hb_V} \\
\Vcal_{\text{HB}}'(\rho)(\theta, r) &= F'(\rho)(\theta) + \frac{1}{2}|r|^2, \label{eq:hb_dV}
\end{align}
corresponding to equations~\eqref{eq:PDE_second_order},\eqref{eq:free_energy} without diffusion.

Viewed as a single-particle dynamics, $v_\text{HB}(\rho)(\cdot)$ describes the damped nonlinear oscillator with potential $F'(\rho)(\cdot)$. It is well-known from the optimization literature that~\eqref{eq:hb_dV} is a Lyapunov function for the dynamics~\eqref{eq:hb_v}, see, e.g.,~\citen{{gadat2018stochastic}}. This fact can be easily verified: for all $\theta, r$,
\begin{align}
\braket{\nabla\Vcal_{\text{HB}}'(\rho)(\theta, r)}{v_{\text{HB}}(\rho)(\theta, r)}
&= \braket{\vect{\nabla F'(\rho)(\theta) \\ r}}{\vect{r \\ -\nabla F'(\rho)(\theta) - \gamma r}} \notag \\
&= -\gamma |r|^2 \leq 0.\notag
\end{align}
Therefore, that $\Vcal_{\text{HB}}$ is a Lyapunov functional for the mean-field dynamics is a simple consequence of the single particle case. Proposition~\ref{prop:lyap_decreasing} is an extension of this fact to the case with diffusion.


\subsection{Time-derivative of the free energy}
\label{app:details}
\begin{proof}[Proof of Proposition~\ref{prop:lyap_decreasing}]
From the expression of the free energy $\Ecal(\rho) = F([\rho]^\theta) + \braket{\frac{1}{2}|r|^2}{\rho} + \frac{1}{\beta}\braket{\log \rho}{\rho}$, we obtain,
\begin{equation}
\label{eq:d_free_energy}
\Ecal'(\rho)(\theta, r) = F'([\rho]^\theta)(\theta) + \frac{1}{2}|r|^2 + \beta^{-1} (1 + \log \rho(\theta, r)),
\end{equation}
and, using the shorthand $\ell_\theta := \beta^{-1}\nabla_\theta \log \rho_t$, $\ell_r := \beta^{-1}\nabla_r \log \rho_t$, we compute
\def\adjmargin{\kern-12pt}
\begin{align*}
\partial_t \Ecal(\rho_t)
&= \braketd{\nabla \Ecal'(\rho_t)}{\rho_t v(\rho_t)} & \text{\adjmargin by Lemma~\ref{lem:variation}}\\
&= \braketd{\vect{\nabla_\theta F'([\rho_t]^\theta) + \ell_\theta \\ r + \ell_r}}{\rho_t\vect{r \\ -\nabla_\theta F'([\rho_t]^\theta) - \gamma r - \gamma\ell_r}} & \text{\adjmargin by~\eqref{eq:noisy_second_order} and~\eqref{eq:d_free_energy}}\\
&= \braket{ - \gamma \dotp{r}{r} - \gamma \dotp{\ell_r}{\ell_r} - 2\gamma \dotp{r}{\ell_r} + \dotp{\ell_\theta}{r} - \dotp{\ell_r}{\nabla_\theta F'([\rho_t]^\theta)} }{ \rho_t}.
\end{align*}
We conclude by showing that the last two terms, $\braket{\braket{\ell_\theta}{r}}{\rho_t}$, and $\braket{\braket{\ell_r}{\nabla_\theta F'([\rho_t]^\theta)}}{\rho_t}$ are equal to zero. Indeed,
\begin{align*}
\int_{\Tcal\Theta} &\braket{\ell_\theta(\theta, r)}{r} \rho_t(\theta, r)\:\differential\theta\differential r\\
&= \beta^{-1} \int_{\Tcal\Theta} \braket{\nabla_\theta \log \rho_t(\theta, r)}{r}\rho_t(\theta, r)\:\differential\theta\differential r &\text{by definition of $\ell_\theta$} \\
&= \beta^{-1} \int_{\Tcal\Theta} \braket{\nabla_\theta \rho_t(\theta, r)}{r} \differential\theta\differential r \\
&= -\beta^{-1} \int_{\Tcal\Theta} \braket{\rho_t(\theta, r)}{\nabla_\theta \cdot r} \differential\theta\differential r & \text{by duality~\eqref{eq:duality}} \\
&= 0,
\end{align*}
and similarly,
\begin{align*}
\int_{\Tcal\Theta} &\braket{\ell_r(\theta, r)}{\nabla_\theta F'([\rho_t]^\theta(\theta))} \rho_t(\theta, r)\:\differential\theta\differential r \\
&= \beta^{-1} \int_{\Tcal\Theta} \braket{\nabla_r \log \rho_t(\theta, r)}{\nabla_\theta F'([\rho_t]^\theta(\theta))} \rho_t(\theta, r)\:\differential\theta\differential r &\text{by definition of $\ell_r$} \\
&= \beta^{-1} \int_{\Tcal\Theta} \braket{\nabla_r \rho_t(\theta, r)}{\nabla_\theta F'([\rho_t]^\theta)(\theta)} \differential\theta \differential r \\
&= -\beta^{-1} \int_{\Tcal\Theta} \braket{\rho_t(\theta, r)}{\nabla_r \cdot \nabla_\theta F'([\rho_t]^\theta)(\theta)} \differential\theta \differential r & \text{by duality~\eqref{eq:duality}} \\
&= 0,
\end{align*}
where the last equality is due to the fact $F'([\rho_t]^\theta)$ does not depend on $r$.
\end{proof}

\subsection{Additional bounds on the entropy and free energy}
\label{app:bounds}
We recall the expression of the free energy:
\begin{align*}
\Ecal(\rho) &= F([\rho]^\theta) + \braket{\frac{1}{2}|r|^2}{\rho} + H(\rho)\\
&= F_0([\rho]^\theta) + \braket{g(\theta) + \frac{1}{2}|r|^2}{\rho} + H(\rho),
\end{align*}
where $H(\rho) := \braket{\log \rho}{\rho}$ is the negative entropy, $F_0(\rho) = R(\braket{\Psi}{\rho})$ is the unregularized risk, and $g: \Theta \to \Rbb_+$ is the regularization function.

Let $\Kcal$ be the set
\begin{equation}
\Kcal := \{\rho \in \Pcal(\Tcal\Theta) : \braket{g(\theta) + |r|^2/2}{\rho} < \infty \}.
\end{equation}
First, we provide the following lower-bound on the free energy. For $\rho\in\Pcal(\Tcal\Theta)$, we write $\log \rho = \log^{+}\rho - \log^{-}\rho$, where $\log^{+}\rho:= \max\{\log\rho,0\}$ and $\log^{-}\rho:= \max\{-\log\rho,0\}$.
\begin{proposition}
\label{prop:lyap_bounded}
Suppose that assumptions \ref{A:R_diff}-\ref{A4} hold. Then there exists a positive function $C(\alpha)$ such that for all $\rho \in \Kcal$, and all $\alpha \leq \beta$
\begin{equation}
\label{eq:free_energy_lower_bound}
\Ecal(\rho)
\geq F_0([\rho]^\theta) +(1 - \alpha/\beta)\braket{g}{\rho} - \frac{C(\alpha)}{\beta}.
\end{equation}
\end{proposition}
\begin{proof}
We can decompose $\Ecal$ into
\begin{equation}
\label{eq:free_energy_decomp}
\Ecal(\rho) = F_0([\rho]^\theta) + \braket{g(\theta) + |r^2|/2}{\rho} + \frac{1}{\beta}\big(\braket{\log^+ \rho}{\rho} - \braket{\log^- \rho}{\rho}\big).
\end{equation}
We focus on bounding the last term. First, following~\citen{[Prop. 2.3]{bouchut1995long}}, observe that for any constant $c \geq 1$, we have
\[
x \log^- x \leq c (x + e^{-c}) \quad\text{for all $x \geq 0$}.
\]
The inequality is trivial for $x \geq 1$ since the LHS is $0$, by definition. For $x \in [0, 1]$, this can be verified by noting that the difference $d(x) := x\log^- x - c(x + e^{-c})$ attains its maximum at $x = e^{-c-1}$, and $d(e^{-c-1}) \leq 0$. Applying the previous inequality with a function $c : \Tcal\Theta \to [1, +\infty)$, we have
\[
\braket{\log^- \rho}{\rho} \leq \braket{c(\theta, r)}{\rho(\theta, r) + e^{-c(\theta, r)}}.
\]
Let $\alpha > 0$ and take $c(\theta, r) := 1 + \alpha (g(\theta) + \frac{|r|^2}{2})$, which is $\geq 1$ since the regularizer $g$ is non-negative by assumption. Then
\begin{equation}
\label{eq:free_energy_lower_bound_proof_0}
\braket{\log^- \rho}{\rho} \leq 1 + \alpha \braket{g(\theta)+|r|^2/2}{\rho}
+ \int_{\Tcal\Theta} (1+\alpha g(\theta) + \alpha \frac{|r|^2}{2}) e^{-1 -\alpha g(\theta)-\alpha \frac{|r|^2}{2}} \dtheta\dr.
\end{equation}
We shall prove that the last term, which we denote by $C(\alpha) := \int_{\Tcal\Theta} (1 + \alpha g + \alpha \frac{|r|^2}{2})e^{-1 - \alpha g(\theta) - \alpha \frac{|r|^2}{2}}$, is finite by virtue of assumption \ref{A4}. Indeed, the assumption guarantees that $e^{-\alpha g}$ is integrable. It also follows that $g e^{-\alpha g}$ is integrable: indeed, for any $\epsilon \in (0, \alpha)$, using the inequality $1 + \epsilon g \leq e^{\epsilon g}$, we can write that $ge^{-\alpha g} \leq \frac{e^{-\alpha g + \epsilon g} - e^{-\alpha g}}{\epsilon}$, and the upper-bound is integrable by assumption~\ref{A4}.

To summarize, we obtain
\begin{equation}
\label{eq:free_energy_lower_bound_proof}
\braket{\log^- \rho}{\rho} \leq \alpha \braket{g(\theta) + |r|^2/2}{\rho} + C(\alpha),
\end{equation}
for a finite, positive function $C(\alpha)$. Using the last inequality in~\eqref{eq:free_energy_decomp}, and the fact $\braket{\log^+ \rho}{ \rho} \geq 0$, we obtain
\begin{align*}
\Ecal(\rho)
&\geq F_0([\rho]^\theta) + \braket{g(\theta) + |r^2|/2}{\rho}(1 - \alpha / \beta) - \frac{C(\alpha)}{\beta}.
\end{align*}
Finally, taking $\alpha \leq \beta$ guarantees that the term $\braket{|r|^2}{\rho}(1-\alpha/\beta)$ is non-negative, and proves the claim~\eqref{eq:free_energy_lower_bound}.
\end{proof}

\begin{proposition}
\label{prop:BoundedQuantitites}
Let $\rho_t\in C\left([0,\infty),\Pcal(\Tcal\Theta)\right)$ be a solution to~\eqref{eq:noisy_second_order} with initial condition $\rho_{0}\in \Pcal(\Tcal\Theta)$, and suppose that $\rho_0$ satisfies assumption~\ref{A:init}.
Then for all $t \geq 0$, the quantities $\Ecal(\rho_t)$, $F_0([\rho_t]^\theta)$, $\langle g(\theta) + |r|^2/2,\rho_{t}\rangle$, $\langle\log^{+}\rho_t,\rho_t\rangle$, are bounded independently of $t$.
\end{proposition}

\begin{proof}
From (\ref{eq:free_energy_decomp}), we have
\[
F_0([\rho_{t}]^{\theta}) + \braket{g(\theta) + |r|^2/2}{\rho_t} + \beta^{-1}\langle\log^{+}\rho_t,\rho_t\rangle = \Ecal(\rho_t) + \beta^{-1}\langle\log^{-}\rho_t,\rho_t\rangle.
\]
The terms on the left-hand-side are non-negative. We upper bound the right-hand-side using \eqref{eq:free_energy_lower_bound_proof}, to obtain 
\begin{align}
F_0\left([\rho_{t}]^{\theta}\right) + \braket{g(\theta) + |r|^2/2}{\rho_{t}}(1 - \alpha / \beta) + \beta^{-1}\langle\log^{+}\rho_t,\rho_t\rangle \leq \Ecal(\rho_t) + \beta^{-1}C(\alpha).
\label{eq:provingboundedmeanrsquared}
\end{align}
Choosing $\alpha < \beta$, as in the proof of Proposition \ref{prop:lyap_bounded}, and using the fact $\Ecal(\rho_t)$ is a decreasing function of~$t$ (Proposition \ref{prop:lyap_decreasing}), we have
\begin{align*}
0
&\leq F_0\left([\rho_{t}]^{\theta}\right) + \braket{g(\theta) + |r|^2/2}{\rho_{t}}(1 - \alpha/\beta) + \beta^{-1}\langle\log^{+}\rho_t,\rho_t\rangle\\
&\leq \Ecal(\rho_t) + \beta^{-1}C(\alpha)\\
& \leq \Ecal(\rho_0) + \beta^{-1}C(\alpha) < \infty
\end{align*}
where the $\Ecal(\rho_0)$ is finite by virtue of assumption~\ref{A:init}. The statement follows.
\end{proof}

The following is a consequence of Propositions~\ref{prop:lyap_decreasing} and~\ref{prop:BoundedQuantitites}.
\begin{theorem}\label{thm:Barabalat}
Consider the set up in Propositions \ref{prop:lyap_decreasing} and \ref{prop:BoundedQuantitites}. Then the solution trajectory $(\rho_t)_{t\geq 0}$ for (\ref{eq:noisy_second_order}) satisfies
\begin{align}
\label{eq:dotofLyap}
\lim_{t\rightarrow\infty} \int_{\Tcal\Theta} |r + \beta^{-1}\nabla_r \log \rho_{t}|^2 \rho_{t}\:\differential\theta\differential r = 0.
\end{align}	
\end{theorem}

\begin{proof}
Propositions \ref{prop:lyap_decreasing} and \ref{prop:BoundedQuantitites} allow us to deduce that the functional $\Ecal(\rho_{t})$ given by (\ref{eq:free_energy}) has a finite limit as $t\rightarrow\infty$. Now our strategy is to prove that $G:=\partial_{t}\Ecal$ is uniformly continuous in $t$. Then, by Barbalat's lemma \citen{[Lemma 4.2]{slotine1991applied}}, the claim (\ref{eq:dotofLyap}) follows.

To prove the uniform continuity of $G$ in $t$, it suffices to show that $\partial_t G$ is upper bounded for all $t\geq 0$. First notice that
\begin{align}
|\partial_t G| = |\langle\nabla G^{\prime}(\rho_{t}),\rho_{t}v_t\rangle| = |\mathbb{E}_{\rho_{t}}\left[\nabla G^{\prime},v_t\right]| \leq \sqrt{\mathbb{E}_{\rho_t}\left[|\nabla G^{\prime}(\rho_t)|^2\right]}\:\sqrt{\mathbb{E}_{\rho_t}\left[|v_t|^2\right]},
\label{eq:CauchySchwarz}	
\end{align}
where the last inequality is due to Cauchy-Schwarz. By applying Cauchy-Schwarz again,
\begin{align}
\mathbb{E}_{\rho_t}\!\!\left[|v_t|^2\right] &= \!\int_{\Tcal\Theta}\!\left(|r|^2 + |\nabla_{\theta}F^{\prime}([\rho_{t}]^{\theta}) + \gamma r + \gamma\ell_{r}|^2\right)\rho_{t}\differential\theta\differential r \nonumber\\
&\leq \!\int_{\Tcal\Theta}\!\left((1+3\gamma^2)|r|^2 + 3\gamma^2 |\ell_{r}|^2 + 3|\nabla_{\theta}F^{\prime}([\rho_{t}]^{\theta})|^{2}\right)\rho_{t}\differential\theta\differential r.
\label{eq:boundingVectorFieldMeanSquaredMagnitude}	
\end{align}
Per Assumption~\ref{A:grad_bounded}, $\nabla_{\theta}F^{\prime}(\rho_{t}^{\theta})(\theta) \in L^{\infty}(\Theta)$, and hence $\int|\nabla_{\theta}F^{\prime}(\rho_{t}^{\theta})|^{2}\rho_{t}\differential\theta\differential r < \infty$. From Proposition \ref{prop:BoundedQuantitites}, we know that $\int |r|^2\rho_{t}\differential\theta\differential r<\infty$. Noting that $\int |\ell_{r}|^2 \rho_{t}\differential\theta\differential r = \beta^{-2}\int \frac{|\nabla_{r}\rho_t|^2}{\rho_t}\differential\theta\differential r = 4\beta^{-2}\int |\nabla_{r}\sqrt{\rho_t}|^2\differential\theta\differential r$, and that $\nabla_{r}\sqrt{\rho}\in L^{2}([0,T],\mathcal{T}\Theta)$ for any $T>0$ (see e.g., \citen{[Lemma 3.10]{soler1997asymptotic}}), we have $\int |\ell_{r}|^2 \rho_{t}\differential\theta\differential r <\infty$. Putting these together, we find that (\ref{eq:boundingVectorFieldMeanSquaredMagnitude}) is finite for all $t\geq 0$. We also note that the finiteness of $\int |\ell_{r}|^2 \rho_{t}\differential\theta\differential r = \beta^{-2}\int \frac{|\nabla_{r}\rho_t|^2}{\rho_t}\differential\theta\differential r$ implies that $\rho_{t}$ is positive almost everywhere, and that the Fisher information $\int \frac{|\nabla\rho_t|^2}{\rho_t}\differential\theta\differential r < \infty$.

To show that the other factor in the right-hand-side of (\ref{eq:CauchySchwarz}) is finite, let $G_{1}:=-\gamma\langle|r|^2,\rho_t\rangle$, $G_{2}:=-2\gamma\int\langle r,\ell_{r}\rangle\rho_{t}\differential\theta\differential r$, $G_{3}:=-\gamma\langle|\ell_r|^2,\rho_t\rangle$, and notice that 
\begin{align}
G = -\gamma\int_{\Tcal\Theta} |r + \ell_r|^{2}\rho_{t}\:\differential\theta\differential r = G_{1} + G_{2} + G_{3}.
\label{eq:G1G2G3}	
\end{align}
Direct calculation of the functional derivatives yield
\def\nablafull{\nabla_{{\tiny{\begin{pmatrix}
\theta\\
r	
\end{pmatrix}}}
}}
\begin{subequations}
\begin{align}
&G_{1}^{\prime} = -\gamma |r|^2, \label{eq:G1prime}\\
&G_{2}^{\prime} = 2\beta^{-1}\gamma \nablafull \cdot \frac{\partial}{\partial\nablafull \rho_t}\bigg\langle\begin{pmatrix}
0\\
r	
\end{pmatrix}, \nablafull \rho_{t}
\bigg\rangle = 2d\beta^{-1}\gamma, \label{eq:G2prime}\\
&G_{3}^{\prime} = -\beta^{-2}\gamma\left(-\frac{|\nabla_{r}\rho_{t}|^{2}}{\rho_{t}^{2}} - \nablafull \cdot \rho_{t}^{-1}\frac{\partial}{\partial\nablafull \rho_t}\bigg\langle\begin{pmatrix}
0\\
\nabla_{r}\rho_{t}	
\end{pmatrix},\nabla_{\tiny{\begin{pmatrix}
\theta\\
r	
\end{pmatrix}}}\rho_{t}
\bigg\rangle\right) \nonumber\\
&\quad = -\beta^{-2}\gamma\left(-\frac{|\nabla_{r}\rho_{t}|^{2}}{\rho_{t}^{2}} - \frac{2}{\rho_{t}}\Delta_{r}\rho_{t} + \frac{2}{\rho_{t}^{2}}|\nabla_{r}\rho_{t}|^{2}\right) = -\beta^{-2}\gamma\left(\frac{|\nabla_{r}\rho_{t}|^{2}}{\rho_{t}^{2}} - \frac{2}{\rho_{t}}\Delta_{r}\rho_{t}\right). \label{eq:G3prime}
\end{align}
\label{eq:PrimeOfG1G2G3}
\end{subequations}
Combining (\ref{eq:G1G2G3}) and (\ref{eq:PrimeOfG1G2G3}), we get $G^{\prime} = -\gamma |r|^2 + 2d\beta^{-1}\gamma - \beta^{-2}\gamma \left(\rho_{t}^{-2}|\nabla_{r}\rho_{t}|^{2} - 2\rho_{t}^{-1}\Delta_{r}\rho_{t}\right)$. Therefore,
\begin{align}
\nabla G^{\prime} = \begin{pmatrix}
 \nabla_{\theta}G^{\prime}\\
  \nabla_{r}G^{\prime}	
 \end{pmatrix} = -\gamma\begin{pmatrix}
\beta^{-2}\nabla_{\theta}\left(\rho_{t}^{-2}|\nabla_{r}\rho_{t}|^{2} - 2\rho_{t}^{-1}\Delta_{r}\rho_{t}\right)\\
2r + \beta^{-2}\nabla_{r}\left(\rho_{t}^{-2}|\nabla_{r}\rho_{t}|^{2} - 2\rho_{t}^{-1}\Delta_{r}\rho_{t}\right)	
\end{pmatrix}.
\label{eq:nablaGprime}	
\end{align}
Recalling that $\beta^{-2}\rho_{t}^{-2}|\nabla_{r}\rho_{t}|^{2} = |\ell_{r}|^{2}$, we get 
\begin{align}
|\nabla G^{\prime}|^{2} \leq 3\gamma^{2}\left(4|r|^{2} + \bigg|\nabla_{\tiny{\begin{pmatrix}
\theta\\
r	
\end{pmatrix}}} |\ell_{r}|^{2}\bigg|^{2} + 4 \beta^{-4}\bigg|\nabla_{\tiny{\begin{pmatrix}
\theta\\
r	
\end{pmatrix}}} \rho_{t}^{-1}\Delta_{r}\rho_{t}\bigg|^{2}\right),
\label{eq:BoundingNablaGPrimeSquared}	
\end{align}
and hence $\mathbb{E}_{\rho_{t}}\left[|\nabla G^{\prime}|^{2}\right]$ (the other factor in the RHS of (\ref{eq:CauchySchwarz})) is less than or equal to
\begin{align}\label{eq:ExpectedNablaGPrimeSquared}	
12\gamma^{2}\underbrace{\int |r|^{2}\rho_{t}\differential\theta\differential r}_{\text{term 1}} + 3\gamma^{2}\underbrace{\int\bigg|\nabla_{\tiny{\begin{pmatrix}
\theta\\
r	
\end{pmatrix}}} |\ell_{r}|^{2}\bigg|^{2}\rho_{t}\differential\theta\differential r}_{\text{term 2}} + 12\beta^{-4}\gamma^{2}\underbrace{\int\bigg|\nabla_{\tiny{\begin{pmatrix}
\theta\\
r	
\end{pmatrix}}} \rho_{t}^{-1}\Delta_{r}\rho_{t}\bigg|^{2}\rho_{t}\differential\theta\differential r}_{\text{term 3}}.
\end{align}
By Proposition \ref{prop:BoundedQuantitites}, the term 1 in (\ref{eq:ExpectedNablaGPrimeSquared}) is finite. Showing the finiteness of the terms 2 and 3 in (\ref{eq:ExpectedNablaGPrimeSquared}) requires somewhat tedious estimates. We only sketch the main ideas for the same. 

Letting $u:=|\ell_{r}|^{2}$, term 2 equals $\int |\nabla u|^{2} \rho_{t} \differential\theta\differential r = \int \langle\nabla u, \rho_{t}\nabla u\rangle \differential\theta\differential r$, which upon integration-by-parts and setting the boundary term to zero becomes: 
\[-\int u\nabla\cdot(\rho_t\nabla u)\differential\theta\differential r = -\int u \left(\langle\nabla\rho_{t},\nabla u\rangle + \rho_t\Delta u\right)\differential\theta\differential r = -\mathbb{E}_{\rho_t}\left[\langle\nabla\log\rho_{t},u\nabla u\rangle\right] + \mathbb{E}_{\rho_t}\left[u\Delta u\right].\]
Thus, term 2 in (\ref{eq:ExpectedNablaGPrimeSquared}) can be written as
\begin{align}
\int |\nabla u|^{2} \rho_{t} \differential\theta\differential r &= |-\mathbb{E}_{\rho_t}\left[\langle\nabla\log\rho_{t},u\nabla u\rangle\right] + \mathbb{E}_{\rho_t}\left[u\Delta u\right]|\nonumber\\
&\leq \mathbb{E}_{\rho_t}\left[|-\langle\nabla\log\rho_{t},u\nabla u\rangle + u\Delta u|\right]\nonumber\\
&\leq \mathbb{E}_{\rho_t}\left[|-\langle\nabla\log\rho_{t},u\nabla u\rangle|\right] + \mathbb{E}_{\rho_t}\left[|u\Delta u|\right]	\nonumber\\
&\leq |\nabla\log\rho_{t}|_{L^{2}(\rho_t)} |u\nabla u|_{L^{2}(\rho_t)} + |u|_{L^{1}(\rho_t)} |\Delta u|_{L^{\infty}(\rho_t)},
\label{HolderIneq}
\end{align}
wherein we used the Jensen's, triangle and H\"{o}lder's inequalities, respectively. Finiteness for two of the four terms in (\ref{HolderIneq}) have been pointed out before: $|\nabla\log\rho_{t}|_{L^{2}(\rho_t)} = \int \frac{|\nabla\rho_t|^2}{\rho_t}\differential\theta\differential r$ (the Fisher information) $< \infty$, and $|u|_{L^{1}(\rho_t)} = \int |\ell_r|^2\rho_t \differential\theta\differential r < \infty$. Following some calculation, the same estimates can be used to bound the remaining two terms.

For term 3 in (\ref{eq:ExpectedNablaGPrimeSquared}), notice that $\rho_{t}^{-1}\Delta_{r}\rho_{t} = \beta^{2}|\ell_r|^2 + \beta\nabla_{r}\cdot\ell_{r}$, and hence term 3 equals 
\[\mathbb{E}_{\rho_t}\bigg|2\beta^2\left(\frac{\partial\ell_r}{\partial\theta}\right)^{\!\!\top}\!\ell_r + \beta\nabla_{\theta}\left(\nabla_r\cdot\ell_r\right) \bigg|^2 + \mathbb{E}_{\rho_t}\bigg| 2\beta^2\left(\frac{\partial\ell_r}{\partial r}\right)^{\!\!\top}\!\ell_r + \beta\nabla_{r}\left(\nabla_r\cdot\ell_r\right) \bigg|^2.\]
Similar estimates as before show the finiteness of the above. We summarize: since each of the two factors in the RHS of (\ref{eq:CauchySchwarz}) are finite, $\partial_t G$ is upper bounded for all $t\geq 0$, which suffices to conclude that $G$ is uniformly continuous in $t$. Then by Barbalat's lemma \citen{[Lemma 4.2]{slotine1991applied}}, (\ref{eq:dotofLyap}) follows.
\end{proof}

\section{Stationary solutions and convergence}
\label{app:stationary}
\def\gibbs{\frac{\exp\left(-\frac{\beta}{2}|r|^2\right)}{Z_1}}

\subsection{Proof of Theorem~\ref{thm:decomposition}}
We seek to prove that any stationary solution $\rho^\star$ decomposes into the product of marginals $\rho^\star = \frac{\exp(-\beta \frac{|r|^2}{2})}{Z_1}[\rho^\star]^\theta$, where $[\rho^\star]^\theta$ is a solution to the Boltzmann equation~\eqref{eq:boltzmann}.
\begin{proof}
Let $\rho_t$ be the solution initialized at $\rho^{\star}$. Since $\rho^{\star}$ is stationary, we must have $\partial_t \Ecal(\rho_t) = 0$, i.e.,
\begin{equation}
\label{eq:char_proof}
\int |r + \beta^{-1}\nabla_r \log \rho^{\star}|^2 \differential \rho^{\star} = 0,
\end{equation}
by Proposition~\ref{prop:lyap_decreasing}.
Let $\rho^{\star}(\theta, r) = \gibbs \eta(\theta, r)$. To prove the first part of the claim, we seek to show that $\eta(\theta, \cdot)$ is a constant for a.e. $\theta$. We have $r + \beta^{-1}\nabla_r \log \rho^{\star} = \beta^{-1} \nabla_r \log \eta$, thus~\eqref{eq:char_proof} yields
\[
\nabla_r \log \eta = 0 \quad \text{$\eta$-a.e.}
\]
This implies that for a.e. $\theta$, the function $\eta(\theta, \cdot)$ is constant on its support; but it is also constant (equal to $0$) outside its support, thus by continuity it is constant on its entire domain. This proves the first part of the claim.

So far, we have shown that there exists $\eta \in \Pcal(\Theta)$ such that $\rho^{\star}(\theta, r) = \gibbs \eta(\theta)$, and we seek to characterize $\eta$. By the continuity equation~\eqref{eq:noisy_second_order}, since $\rho^\star$ is stationary, we must have $\nabla \cdot (\rho^\star v(\rho^\star)) = 0$, i.e.,
\begin{align*}
0
&= \nabla_\theta \cdot (\rho^{\star} r) + \nabla_r \cdot [\rho^{\star} (-\nabla F'(\eta) - \gamma r - \gamma\beta^{-1} \nabla_r \log \rho^{\star})] \\
&= \nabla_\theta \cdot (\rho^{\star} r) - \nabla_r \cdot [\rho^{\star} \nabla F'(\eta)]&\text{since } \nabla_r \log \rho^{\star} = -\beta r, \\
&= \braket{\gibbs r}{\nabla_\theta \eta} - \braket{\eta \nabla_{\theta} F'(\eta)}{\nabla_r \gibbs} \\
&= \braket{\gibbs r}{\nabla_\theta \eta + \beta \eta \nabla_{\theta} F'(\eta)},
\end{align*}
where the equality is for a.e. $r, \theta$. Therefore, $\nabla_\theta \eta + \beta \eta \nabla_{\theta} F'(\eta) = 0$. This is equivalent to
\[
\nabla_\theta \log \eta + \beta \nabla_{\theta} F'(\eta) = 0, \ \text{$\eta$-a.e.,}
\]
and integrating, we obtain: $\log \eta = - \beta \nabla_{\theta} F'(\eta) + \text{a constant}$. This is equivalent to the Boltzmann fixed point equation~\eqref{eq:boltzmann}, as desired.
\end{proof}

\subsection{Proof of Proposition~\ref{prop:boltzmann}}
We seek to prove that the operator
\[
T: \rho \mapsto T(\rho) = \frac{\exp(-\beta F'(\rho))}{Z_2(\rho)}
\]
admits a unique point on $\Pcal(\Theta)$. Note that $T$ is well-defined for all $\rho$ by virtue of assumption~\ref{A4}. Indeed, $F'(\rho) = F_0'(\rho) + g$, and assumption~\ref{A4} states that $F'_0$ is uniformly bounded on $\Pcal$, and $\exp(-\beta g)$ is integrable, thus if $M$ is an upper bound on $\|F'_0(\rho)\|_\infty$, we have $\exp(-\beta F'(\rho)) \leq \exp{\beta M} \exp(-\beta g)$, which is integrable.

\begin{proof}[Proof of existence:]
We will use Schauder's fixed point theorem \citen{[p. 286, Theorem 11.6]{gilbarg2015elliptic}}, stated below. Recall that a subset of a metric space is precompact if any sequence in that subset has a converging subsequence.
\begin{theorem*}[Schauder's fixed point theorem]
Let $X$ be a Banach space and $M \subset X$ be non-empty, convex and closed. If $T: M \mapsto M$ is a continuous operator such that $T(M)$ is precompact, then $T$ has a fixed point.
\end{theorem*}

Let $B = \{\rho \in L^1(\Theta) : \rho \geq 0, \|\rho\|_1 \leq 1\}$. Note that $T(B) \subset \Pcal(\Theta)$, since $T(\rho)$ is normalized. Thus, to prove that $T$ has a fixed point on $\Pcal$, it suffices to prove that $T$ has a fixed point on $B$. To this end, we apply Schauder's theorem with $X = L^1(\Theta)$ and $M = B$.

Let $(\rho_n)$ be a sequence of elements in $B$. We shall prove that $(T(\rho_n))$ has a converging subsequence. We have that $F'$ decomposes into $F'(\rho_n) = F'_0(\rho_n) + g$. By assumption \ref{A4}, $(F'_0(\rho_n))_n$ is uniformly equicontinuous and uniformly bounded, thus by the Arzela-Ascoli theorem, there exists a subsequence $(F'_0(\rho_{k_n}))$ that converges uniformly to some continuous, bounded function $\ell$. We will show that $T(\rho_{k_n})$ converges in $L^1$ to $\rho := \exp(-\beta (\ell + g)) / \|\exp(-\beta(\ell + g))\|_1$, which is well-defined since $\ell$ is bounded and $g$ is confining. Observing that for all $n$, $\|T(\rho_{k_n})\|_1 = \|\rho\|_1 = 1$, we have by Scheff\'{e}'s lemma \citen{[p. 55]{williams1991probability}} that poitwise convergence of $T(\rho_{k_n})$ to $\rho$ implies convergence in $L^1$. Thus it suffices to prove pointwise convergence.

By continuity of the exponential function, we have
\begin{equation}
\label{eq:proof_fixed_point_1}
\exp( -\beta F'(\rho_{k_n})) \to \exp(-\beta(\ell+g)),
\end{equation}
where the convergence is pointwise. By assumption \ref{A4}, $\{F'_0(\rho), \rho \in B\}$ is uniformly bounded and $\exp( -\beta g)$ is integrable, thus, by the dominated convergence theorem,
\begin{equation}
\label{eq:proof_fixed_point_2}
\|\exp( -\beta F'(\rho_{k_n}))\|_1 \to \|\exp( -\beta(\ell + g))\|_1.
\end{equation}
By~\eqref{eq:proof_fixed_point_1} and~\eqref{eq:proof_fixed_point_2}, we have $T(\rho_{k_n})$ converges pointwise to $\rho$, which concludes the proof.

\end{proof}

\begin{proof}[Proof of uniqueness:]
Suppose $\rho_1, \rho_2 \in \Pcal(\Theta)$ are two fixed points of $T$. Then we have for $i \in \{1, 2\}$,
\begin{equation}
\label{eq:uniqueness_proof}
\log Z(\rho_i) = -\log \rho_i(\theta) - \beta F'(\rho_i)(\theta) \text{ for a.e. $\theta$}.
\end{equation}
We have
\newcommand\kl[2]{D_{\text{KL}}(#1 \| #2)}
\begin{align*}
0 &= \braket{\log Z(\rho_1) - \log Z(\rho_2)}{\rho_1 - \rho_2} & \text{since $\log Z(\rho_i)$ are constants}\\
&= -\braket{\log \rho_1 + \beta F'(\rho_1) - \log \rho_2 - \beta F'(\rho_2)}{\rho_1 - \rho_2} & \text{by~\eqref{eq:uniqueness_proof}}\\
&= -\braket{F'(\rho_1) - F'(\rho_2)}{\rho_1 - \rho_2} - \kl{\rho_1}{\rho_2} - \kl{\rho_2}{\rho_1} \\
&\leq - \kl{\rho_1}{\rho_2} - \kl{\rho_2}{\rho_1} & \text{by convexity of $F$}.
\end{align*}
where $\kl{\rho_1}{\rho_2} = \braket{\log \frac{\rho_1}{\rho_2}}{\rho_1}$. Note that $\rho_1, \rho_2$ are both normalized by assumption, so both KL divergences are non-negative, with equality if and only if $\rho_1 = \rho_2$ a.e. This concludes the proof.
\end{proof}

\subsection{Proof of Theorem~\ref{thm:stationarymeasure}}
\begin{proof}
(i) Recall from Section 2.2 that under the stated conditions on the initial measure $\mu_{0}$, the equation (\ref{eq:noisy_second_order}) admits a unique solution $(\mu_{t})_{t\geq 0}$ satisfying $\mu_{t} \in C\left([0,\infty),\Mcal(\Tcal\Theta)\right)$, that is, $(\mu_{t})_{t\geq 0}$ is a continuous measure-valued trajectory satisfying $\int\differential\mu_{t} < \infty$ for all $t\geq 0$.

From Proposition~\ref{prop:BoundedQuantitites}, we know that the quantities $F_0([\mu_{t}]^{\theta}) < \infty$, $\int_{\Tcal\Theta} (g(\theta) + |r|^{2}/2)\differential\mu_{t}< \infty$, $\int_{\Tcal\Theta}\log^{+}\mu_{t}\differential\mu_{t}< \infty$ for all $t\geq 0$, with their upper bounds being independent of $t$. Hence by the Dunford-Pettis theorem \citen{[p. 123]{albiac2006topics}}, the solutions $(\mu_{t})_{t\geq 0}$ are weakly compact in $L^{1}(\Tcal\Theta)$. Thus, there exists $\mu^{\star}$ and a subsequence $(\mu_{t_{k}})_{k\geq 1}$ such that $(\mu_{t_{k}})_{k\geq 1}$ converges weakly to $\mu^{\star}$. 

To prove $\mu_{t}$ is absolutely continuous (w.r.t. the Lebesgue measure) for each $t\geq 0$, we now show that the sequence of random vectors $(X_{k})_{k\geq 1} := (\theta_{t_{k}}, r_{t_{k}})_{k\geq 1}$ are uniformly integrable. By de la Vall\'{e}e-Poussin's criterion \citen{[p. 3-4]{chaumont2012exercises}}, the latter holds if and only if there is an increasing function $\Phi : \mathbb{R}_{>0} \mapsto \mathbb{R}_{>0}$ satisfying $\lim_{x\rightarrow\infty}\frac{\Phi(x)}{x} = +\infty$, such that $\sup_{k\geq 1}\mathbb{E}_{\mu_{t_{k}}}\left[\Phi(X_{k})\right] < \infty$. To apply this in our context, we set $\Phi(x)\equiv x^{2}$, and use the result from Proposition~\eqref{prop:BoundedQuantitites} that $\int_{\Tcal\Theta} |r_{t_{k}}|^{2}\differential\mu_{t_{k}}$ is uniformly upper bounded for all $k\geq 1$. Therefore, $(X_{k})_{k\geq 1}$ are uniformly integrable, and equivalently, the measures $\mu_{t_{k}}$ are absolutely continuous, and the corresponding joint PDFs exist for all $k\geq 1$. Taking $\{t_{k}\}_{k\geq 1}$ to be an arbitrary sequence, we deduce that $\mu_{t}$ is absolutely continuous for each $t\geq 0$. Taking $\{t_{k}\}_{k\geq 1}$ to be the sequence corresponding to the weakly convergent subsequence $(\mu_{t_{k}})_{k\geq 1}$ mentioned in the previous paragraph, we deduce that $\mu^{\star}$ is absolutely continuous. 

(ii) Let us consider the joint PDF trajectory $(\rho_{t})_{t\geq 0}$ corresponding to the measure-valued trajectory $(\mu_{t})_{t\geq 0}$ that solves (\ref{eq:noisy_second_order}). From part (i), we know that $(\rho_{t})_{t\geq 0}$ exists and is weakly compact in $L^{1}(\mathcal{T}\Theta)$. Letting
\[\zeta_{t}(s,\theta,r):= \rho_{t+s}(\theta,r),\] 
we now prove that $(\zeta_{t})_{t\geq 0}$ is strongly compact in $C([0,T],L^{1}(\mathcal{T}\Theta))$ for any $T>0$. From Theorem~\ref{thm:Barabalat}, we can write
\begin{align}
\lim_{t\rightarrow\infty}\int_{0}^{T}\partial_{t}\mathcal{E}(t+s)\:\differential s = 0,
\label{eq:LyapDerivative}	
\end{align}
which combined with Proposition~\ref{prop:lyap_decreasing} yields
\begin{align}
\lim_{t\rightarrow\infty} \|r\sqrt{\zeta_{t}} + 2\beta^{-1}\nabla_{r}\sqrt{\zeta_{t}}\|_{L^{2}\left([0,T]\times\mathcal{T}\Theta\right)}	 = 0.
\label{eq:zetalimit}	
\end{align}
The remaining proof follows the same line of arguments as in \citen{[p. 1365--1367]{soler1997asymptotic}}. Specifically, for any given sequence $\{t_{k}\}_{k\geq 1}$ with $\lim_{k\rightarrow\infty}t_{k} = \infty$, letting $\zeta_{k}:=\zeta_{t_{k}}$, one shows that the sequence of $\{\zeta_{k}\}_{k\geq 1}$ is relatively compact in $C\left([0,T],L^{1}(\mathcal{T}\Theta)\right)$ for any $T>0$. Consequently, $\zeta_{k} \rightarrow \zeta_{\infty}$ strongly in $L^{1}(\mathcal{T}\Theta)$, which is to say $\rho_{t} \rightarrow \rho^{\star}$ (equivalently, $\mu_{t} \rightarrow \mu^{\star}$) strongly in $L^{1}(\mathcal{T}\Theta)$, as desired.	
\end{proof}


\subsection{Proof of Theorem~\ref{thm:minimum}}
We start by showing that $\rho^\star$ is a minimizer of $\Ecal(\rho)$ over $\Kcal$, by adapting the argument from the first-order case~\citen{[Lemma 6.2]{mei2018mean}}. We omit some details and emphasize the differences.

Recall that
\[
\Kcal = \{\rho \in \Pcal(\Tcal\Theta) : \braket{g(\theta) + |r|^2/2}{\rho} < \infty\}.
\]
\begin{lemma}
\label{lem:free_energy_minimizer}
Let $\rho^\star$ be the unique solution of the Boltzmann fixed point equation~\eqref{eq:boltzmann}. Then for all $\rho \in \Kcal$, $\Ecal(\rho) \geq \Ecal(\rho^\star)$.
\end{lemma}
\begin{proof}
First, we argue that $\Ecal$ has a minimizer over $\Kcal$. Note that $\Ecal(\rho)$ is lower-bounded on $\Kcal$ by Proposition~\ref{prop:lyap_bounded}. Thus, $\inf_{\rho \in \Kcal} \Ecal(\rho)$ is finite and there exists a sequence $\rho_k \in \Kcal$ such that $\lim_{k \to \infty} \Ecal(\rho_k) = \inf_{\rho \in \Kcal} \Ecal(\rho)$. Furthermore, by the same argument as the proof of Proposition~\ref{prop:BoundedQuantitites}, the quantities $F_0([\rho_k]^\theta), \braket{g(\theta) + |r|^2/2}{\rho_k}, \braket{\log^+ \rho_k}{\rho_k}$ are bounded uniformly in $k$. Thus, by de la Vall\'{e}e-Poussin's criterion~\citen{[p. 3-4]{chaumont2012exercises}}, there exists $\rho_\infty \in \Kcal$ such that a subsequence of $\rho_k$ converges weakly to $\rho_\infty$. By lower semi-continuity of $\Ecal$, we have $\Ecal(\rho_\infty) = \inf_{\rho \in \Kcal} \Ecal(\rho)$.

Second, we show that any minimizer of $\Ecal$ on $\Kcal$, must, in fact, be equal to $\rho^\star$. Let $\bar \rho$ be such a minimizer. Then $\bar \rho$ must be positive a.e., otherwise, a perturbation of $\bar \rho$ can decrease the value of $\Ecal$. Indeed, suppose that there exists a bounded subset $S$ of positive Lebesgue measure, such that $\bar \rho \equiv 0$ on $S$, and define $\bar \rho_\epsilon = (1-\epsilon) \bar \rho + \epsilon u_S$, where $u_S = 1_S / \|1_S\|_1$ is the uniform distribution over $S$. Then $\bar \rho_\epsilon$ is in $\Kcal$ (since $S$ is bounded), and there exist constants $A_0$ and $B_0$ such that
\begin{align*}
F([\bar \rho_\epsilon]^\theta) &\leq (1-\epsilon) F([\bar \rho]^\theta) + \epsilon A_0 & \text{by convexity of $F$}\\
\braket{|r|^2/2}{\bar \rho_\epsilon} &\leq (1-\epsilon)\braket{|r|^2/2}{\rho} + \epsilon B_0 & \text{by boundedness of $S$}\\
H(\bar \rho_\epsilon)
&= \braket{\log ((1-\epsilon) \bar \rho + \epsilon u_S)}{(1-\epsilon) \bar \rho + \epsilon u_S} \\
&\leq (1-\epsilon) H(\bar \rho) + \epsilon \log \frac{\epsilon}{\|1_S\|_1}.
\end{align*}
Summing the previous inequalities, we see that there exists a constant $C$ such that $\Ecal(\bar \rho_\epsilon) \leq (1-\epsilon)\Ecal(\bar \rho) + \epsilon C + \frac{1}{\beta} \epsilon \log(\epsilon)$, which is strictly less than $\Ecal(\bar \rho)$ for $\epsilon < e^{-\beta C}$, a contradiction. Therefore $\bar \rho$ must be positive a.e.

Once we have established that $\bar \rho$ is positive a.e., we can show that $\bar \rho$ satisfies the Boltzmann fixed point equation~\eqref{eq:boltzmann}. Indeed, consider the set $\Gamma_k := \{(\theta, r): \frac{1}{k} \leq \bar \rho(\theta, r) \leq k\}$, and let $\Tcal_k = \{f \in C^\infty(\Tcal\Theta) : \text{support}(f) \subseteq \Gamma_k, \|f\|_\infty \leq 1, \int f = 0\}$. In other words, $\Tcal_k$ is a set of tangent vectors such that $\bar \rho + \frac{1}{k} \Tcal_k \subset \Kcal$. The directional derivative of $\Ecal$ in the direction $f \in \Tcal_k$ is well-defined and given by
\begin{equation}
\label{eq:directional_derivative}
\lim_{\epsilon \to 0} \frac{\Ecal(\bar \rho + \epsilon f) - \Ecal(\bar\rho)}{\epsilon} = \braket{F'([\bar \rho]^\theta) + \frac{1}{2}|r|^2 + (1 + \log \bar \rho)}{f}
\end{equation}
and since $\bar \rho$ is a minimizer of $\Ecal$ on $\Kcal$,~\eqref{eq:directional_derivative} must be non-negative for all $f$. Therefore one must have that the integrand $F'([\bar \rho]^\theta) + \frac{1}{2}|r|^2 + (1 + \log \bar \rho)$ is zero a.e. on $\Gamma_k$. But since $\Tcal\Theta = \cup_{k \geq 1} \Gamma_k$, it must be zero a.e. on $\Tcal\Theta$. This implies that $\bar \rho$ is a solution to the Boltzmann fixed point equation~\eqref{eq:boltzmann}, which admits a unique solution $\rho^\star$ by Proposition~\ref{prop:boltzmann}. This concludes the proof.
\end{proof}

\begin{proof}[Proof of Theorem~\ref{thm:minimum}]
Let $F_\lambda$ denote the regularized functional with regularization coefficient $\lambda$, i.e. $F_\lambda(\rho) = F_0(\rho) + \lambda \braket{g}{\rho}$. We shall prove that there exists a constant $C_1$ such that, for all $\beta \geq 1$,
\[
F_{1 - 1/\beta}([\rho^\star]^\theta) \leq \inf_{\eta \in \Pcal(\Theta)}F(\eta) + \frac{C_1 + d \log \beta}{\beta}.
\]

By Lemma~\ref{lem:free_energy_minimizer}, we have $\Ecal(\rho^\star) \leq \Ecal(\rho)$ for all $\rho \in \Kcal$, and observing that $\rho^\star \in \Kcal$, we have by Proposition~\ref{prop:lyap_bounded} applied to $\rho^\star$ and $\alpha = 1$, $F_{1 - 1/\beta}([\rho^\star]^\theta) \leq \Ecal(\rho^\star) + C(1)/\beta$. Combining the previous bounds, we have for all $\rho \in \Kcal$,
\begin{equation}
\label{eq:minimum_proof_1}
F_{1 - 1/\beta}([\rho^\star]^\theta) \leq \Ecal(\rho) + \frac{C(1)}{\beta}.
\end{equation}
In order to conclude, we shall bound the difference between $\Ecal$ and $F$. Note that $\Ecal(\rho) - F([\rho]^\theta) = \frac{1}{2}\braket{|r|^2}{\rho} + \frac{1}{\beta}\braket{\log \rho}{\rho}$, which can be arbitrarily large due to the entropy term. To resolve this issue, one can take a convolution with a Gaussian to control the entropy. More precisely, let $\eta \in \Pcal(\Theta)$, and define $\rho_\eta \in \Kcal$ as the product:
\[
\rho_\eta(\theta, r) := [g_1*\eta](\theta)g_2(r)
\]
where $g_1, g_2$ are two Gaussian PDFs over $\theta$ and $r$ respectively, each with mean $0$ and variance $1/\beta$, and $*$ denotes the convolution. Our goal is to bound the difference between $\Ecal(\rho_\eta)$ and $F(\eta)$. Following the same line of argument as in~\citen{[Lemma 6.5]{mei2018mean}}, there exists a constant $K$ such that
\begin{align*}
&F(g_1*\eta) \leq F(\eta) + \frac{K}{\beta}, \\
&\braket{\frac{1}{2}|r|^2}{g_2} = \frac{d/2}{\beta}, \\
& H(\rho_\eta) \leq (H(g_1) + H(g_2)) = -d \log(2\pi e / \beta).
\end{align*}
Summing the inequalities above, we obtain
\begin{equation}
\label{eq:minimum_proof_2}
\Ecal(\rho_\eta) = F(g_1*\eta) + \frac{1}{2}\braket{|r|^2}{g_2} + \frac{1}{\beta}H(\rho_\eta) \leq F(\eta) + \frac{K + d/2 - d \log(2\pi e/ \beta)}{\beta}.
\end{equation}
Finally, we combine the inequalities~\eqref{eq:minimum_proof_1} and~\eqref{eq:minimum_proof_2}, to obtain, for all $\eta \in \Pcal(\Theta)$,
\[
F_{1-1/\beta}([\rho^\star]^\theta) \leq F(\eta) + \frac{C_1 + d \log(\beta)}{\beta},
\]
where $C_1$ is a constant equal to $C(1) + K + d/2 - d\log(2\pi e)$. Taking the infimum over $\eta$ yields the desired result.
\end{proof}

\section{The case of quadratic loss}
\label{app:quadratic}
In this section, we illustrate the assumptions in the quadratic loss case.

Let $R$ be given by $R(\psi) = \frac{1}{2} \mathbb E_{(x, y) \sim D} (\psi(x) - y)^2 = \frac{1}{2}\|\psi - y\|_\Fcal^2$, where $(x, y)$ are the input feature and labels, respectively, and $D$ is the joint data distribution. The functional $F$ is the sum $F(\mu) = F_0(\mu) + \braket{g}{\mu}$, where for $\mu \in \Mcal(\Theta)$,
\begin{align}
F_0(\mu)
&= R(\braket{\Psi}{\mu}) \notag\\
&= \frac{1}{2}\|\braket{\Psi}{\mu} - y\|_\Fcal^2 \notag\\
&= \frac{1}{2}\|\braket{\Psi}{\mu}\|^2_\Fcal - \braket{\braket{\Psi}{\mu}}{y}_\Fcal + \frac{1}{2}\|y\|_\Fcal^2. \label{eq:F_quadratic}
\end{align}
The first term in~\eqref{eq:F_quadratic} can be written as
\begin{align*}
\frac{1}{2}\|\braket{\Psi}{\mu}\|^2_\Fcal
= \frac{1}{2}\Exp_{x} \left(\int_\Theta \Psi(\theta)(x) \dmu(\theta)\right)^2 
= \frac{1}{2}\iint_\Theta \Exp_{x} [\Psi(\theta)(x)\Psi(\tilde \theta)(x)] \dmu(\theta)\dmu(\tilde\theta) 
= \frac{1}{2}U[\mu, \mu],
\end{align*}
where $U(\theta, \tilde \theta) := \Exp_x [\Psi(\theta)(x)\Psi(\tilde \theta)(x)]$, 
and the symbol $U[\mu, \nu]$ denotes the double integral $\iint U(\theta, \tilde \theta)\dmu(\theta)\differential\nu(\tilde \theta)$.
The second term in~\eqref{eq:F_quadratic} can be written as
\[
- \braket{\braket{\Psi}{\mu}}{y}_\Fcal = - \Exp_{(x, y)} \left[y \int_\Theta \Psi(\theta)(x) \dmu(\theta)\right] = \braket{V}{\mu},
\]
where $V(\theta) := -\Exp_{x, y}[y\Psi(\theta)(x)]$. The last term in~\eqref{eq:F_quadratic} is a constant independent of $\mu$. To summarize, the functional $F_0$ can be written as
\begin{equation}
\label{eq:F_quadratic_unreg}
F_0(\mu) = \frac{1}{2}U[\mu, \mu] + \braket{V}{\mu} + \frac{1}{2}\|y\|_\Fcal^2.
\end{equation}
We now discuss our assumptions in this quadratic case. In particular, we show that the assumptions made in~\citen{{mei2018mean}} (for the first-order gradient flow) imply our assumption \ref{A4}.

First, it is assumed in~\citen{{mei2018mean}} that a quadratic regularizer is used, $g(\theta) = |\theta|^2/2$, which is confining since $\lim_{|\theta| \to \infty} g(\theta) = \infty$ and $\exp(-\beta g)$ is integrable for all $\beta > 0$. This proves the second part of assumption~\ref{A4}. They also make the following assumptions on $U, V$.
\begin{enumerate}[label=(B\arabic*)]
\item \label{B:bounded} $U$ and $V$ are uniformly bounded i.e., there exists $C_1,C_2>0$ such that
\begin{align*} 
    \|U(\cdot,\cdot)\|_\infty \leq C_1, \quad  \|V(\cdot)\|_\infty \leq C_2.
\end{align*}
\item \label{B:boundedgradient} $U$ and $V$ are differentiable, and have bounded gradients, i.e., there exist $C_3, C_4$ such that
\begin{align*} 
    \lVert \nabla U(\cdot,\tilde{\theta}) \rVert_{\infty}\leq C_4 \text{ for all}\: \tilde \theta, \quad \lVert \nabla V(\cdot) \rVert_{\infty} \leq C_3.
\end{align*}
\end{enumerate}
To prove that~\ref{A4} is satisfied, we need to show that the family $\{F_0'(\rho), \rho \in B\}$ is uniformly equicontinuous and uniformly bounded, where $B = \{\rho \in L^1(\Theta), \|\rho\|_1 \leq 1\}$. From~\eqref{eq:F_quadratic_unreg}, the Fr\'echet differential of $F_0$ is given by
\[
F_0'(\rho)(\theta) = U[\rho](\theta) + V(\theta),
\]
where the symbol $U[\rho]$ denotes the function $U[\rho](\theta) = \int U(\theta, \tilde \theta) \rho(\tilde\theta) \dtheta$. Then,
\begin{itemize}
\item By \ref{B:bounded}, $U[\rho] + V$ is bounded, uniformly in $\rho \in B$.
\item From \ref{B:boundedgradient}, it also follows that for $U[\rho] + V$ is Lipschitz continuous, with a Lipschitz constant independent of the choice of $\rho \in B$ and thus the family $\left\{ U[\rho] + V, \rho \in B \right\}$ is uniformly equicontinuous.
\end{itemize}

Thus assumption \ref{A4} is satisfied.

Note that in~\citen{{mei2018mean}}, the regularization term $g(\theta) = \braket{|\theta|^2/2}{\mu}$, together with the boundedness assumptions~\ref{B:bounded}-\ref{B:boundedgradient}, are crucial to guarantee integrability of $\exp(-\beta F'(\mu))$, so that the Boltzmann distribution~\eqref{eq:boltzmann} is well-defined. In the linear case described in Section~\ref{sec:linear}, it is also common to assume that the potential (which in this case is the same as our regularizer $g$) is confining, see for example~\citen{[Definition 4.2]{pavliotis2014stochastic}}.

Assumption \ref{A4} generalizes the conditions on $F_0$ from the quadratic setting to the convex setting, and replaces the quadratic regularizer with a more general confining regularizer.


\end{document}